%% file: main.tex
\documentclass[12pt,reqno,sumlimits]{amsart}

\usepackage{amssymb,amscd,amsmath,epsfig}
\input newheading.tex

\usepackage[top=3cm, bottom=1.8cm, left=2cm, right=2cm,headsep=30pt]{geometry} 
\usepackage{tikz-cd}

 \newcommand{\dir}{\partial\kern-.570em /}
 \newcommand{\dire}{\partial\kern-.570em /{}^{\rm eq}}

\begin{document}

\title{$S^1$-Equivariant Chern-Weil constructions on loop space}
\author[T. McCauley]{Thomas McCauley}
\address{Department of Mathematics and Statistics\\
  Boston University}
\email{tmccaule@math.bu.edu}

\maketitle

\begin{abstract}
We study the existence of $S^1$-equivariant characteristic classes on certain natural infinite rank bundles over the loop space $LM$ of a manifold $M$.  We discuss the different $S^1$-equivariant cohomology theories in the literature and clarify their relationships.  We attempt to use $S^1$-equivariant Chern-Weil techniques to construct $S^1$-equivariant characteristic classes.  The main result is the construction of a sequence of $S^1$-equivariant characteristic classes on the total space of the bundles, but these classes do not descend to the base $LM$.  Nevertheless, we conclude by identifying a class of bundles for which the $S^1$-equivariant first Chern class does descend to $LM$.
\end{abstract}

\bigskip

\section{Introduction}

In this paper we study the existence of $S^1$-equivariant characteristic classes on certain natural infinite rank bundles over the loop space $LM$ of a manifold $M$.  Our main result is the construction by $S^1$-equivariant Chern-Weil techniques of an $S^1$-equivariant first Chern class associated to a structure group reduction of these bundles.  $S^1$-equivariant characteristic classes on $LM$ have attracted interest for many years, going back to Witten's formal proof of the index theorem by formally applying finite dimensional $S^1$-equivariant techniques to $LM$ \cite{Atiyah} and Bismut's construction of the Bismut-Chern character \cite{Bis}.  These $S^1$-equivariant characteristic classes belong to different $S^1$-equivariant cohomology theories.  As a first task, we summarize these $S^1$-equivariant cohomology theories and we clarify how they are related.  One particular theory, $H_{S^1}^*(N)$, is distinguished by having a topological model, as discussed in Section \ref{EquivariantCohomology}.  In Section \ref{LoopSpaceConstructions} we find that $S^1$-equivariant Chern-Weil techniques only partially extend to the pushdown bundles we consider.  Our main result in this section, Theorem \ref{IndependentOfConnection}, proves that these techniques define $S^1$-equivariant characteristic classes on the total space of an associated principal bundle that do not descend to the base.  We then identify a class of bundles for which one of these classes does descend to $LM$, defining an $S^1$-equivariant first Chern class associated to these infinite rank bundles.  Theorem \ref{FCCExtension} proves that this class extends the ordinary first Chern class on $M$.

Characteristic classes on the loop space of a manifold have been studied in a variety of contexts.  Ordinary characteristic classes have been considered in \cite {McLaughlin}, for example, where McLaughlin showed that $M$ admits a string structure precisely when a certain characteristic class on $LM$ vanishes.  Moreover, characteristic classes on $LM$ have informed the study of 2-dimensional field theories on $M$, known as sigma models, by regarding them as 1-dimensional field theories on $LM$.  For example, to study fermions one asks for a spin structure on $LM$, which is a certain lift of the structure group of the frame bundle of $LM$; see \cite{Waldorf} for a discussion of these ideas.

Because $LM$ admits an $S^1$-action by rotation of loops, it is natural to study the $S^1$-equivariant cohomology of $LM$.  $S^1$-equivariant characteristic classes were studied in \cite{Atiyah} where, following an idea of Witten \cite{Witten}, Atiyah showed that one can formally compute the index of the Dirac operator on the spin complex of a spin manifold $M$ as an integral of certain $S^1$-equivariant characteristic classes over $LM$.  Exploring this idea, Bismut \cite{Bis} defined the Bismut-Chern character, ${\rm BCh}$, a differential form on $LM$ that extends the Chern character, ${\rm Ch}$.  The definition of ${\rm BCh}$ was refined in \cite{GJP} using methods from non-commutative geometry, and has been studied further, for example in \cite{TWZ}.  Recently a twisted Bismut-Chern character was defined in \cite{HM} and used to study T-duality in type IIA and IIB string theory from a loop space perspective.

In Section \ref{EquivariantCohomology} we present the different $S^1$-equivariant cohomology theories used in the literature.  In the finite dimensional setting, $S^1$-equivariant characteristic classes on an $S^1$-manifold $N$ belong to $H_{S^1}^*(N)$, the $S^1$-equivariant cohomology of $N$, which we recall in Section \ref{CartanModel}.  Completed periodic $S^1$-equivariant cohomology, $h_{S^1}^*(N)$ is defined in Section \ref{PeriodicCohomology}.  In Section \ref{SuperCohomology} we introduce an $S^1$-equivariant cohomology theory $\bar{h}_{S^1}^*(N)$, which we call super $S^1$-equivariant cohomology. $\bar{h}_{S^1}^*(N)$ has been used in the literature and was called Witten's complex in \cite{AB}.  The table and diagram in Section \ref{LocalizationThms} summarizes these $S^1$-equivariant cohomology theories and the maps between them.  For finite dimensional manifolds, $S^1$-equivariant characteristic classes may be constructed by $S^1$-equivariant Chern-Weil techniques.  There are two equivalent constructions, one via $S^1$-equivariant vector bundles, outlined in Section \ref{GVectorBundle}, and another by $S^1$-equivariant principal bundles, outlined in Section \ref{GPrincipalBundle}.

In the study of $S^1$-equivariant characteristic classes on loop space, one may ask whether $S^1$-equivariant Chern-Weil techniques can be used to construct $S^1$-equivariant characteristic classes.  Such an approach is hinted at in \cite{Bis} and explicitly attempted in \cite{LMRT}.  In Section \ref{PushdownBundles} we construct the pushdown bundle $\mathcal{E} \to LM$, an infinite rank vector bundle built from a finite rank vector bundle $E \to M$, whose fiber is modeled on $L\mathbb{C}^n$.  In particular we see that the pushdown bundle is an $S^1$-equivariant vector bundle, and we ask whether we can construct $S^1$-equivariant characteristic classes.  Section \ref{CurvatureOperators} summarizes the attempt to construct an $S^1$-equivariant Chern character via a covariant derivative on the pushdown bundle, as in \cite{LMRT}, and we note why the construction is not well defined.  The reason this construction fails becomes clearer when we attempt to construct the $S^1$-equivariant Chern character via the principal $LU(n)$-bundle $LFrE$, the loop space of the frame bundle $FrE \to M$, which serves as the frame bundle for $\mathcal{E}$.  In Section \ref{PrincipalBundleApproach} we see that $LFrE$ is not an $S^1$-equivariant principal $LU(n)$-bundle.  Consequently, our main result, Theorem \ref{IndependentOfConnection}, constructs a sequence of $S^1$-equivariant characteristic classes on the total space $LFrE$ that do not descend to the base $LM$.  For comparison, Section \ref{BismutChernCharacter} summarizes the construction of ${\rm BCh}$, and we show that this characteristic class descends to $LM$ but belongs to $\bar{h}_{S^1}^*(LM)$ rather than $H_{S^1}^*(LM)$.

We end by discussing in Section \ref{EquFCC} a class of bundles for which the $S^1$-equivariant Chern-Weil techniques define an $S^1$-equivariant first Chern class that descends to $LM$, after passing to a reduction of the structure group of $LFrE$.  In Section \ref{StructureGroupReduction} we show that $LFrE$ admits a reduction of its structure group to $L^0U(n)$, the connected component of $LU(n)$ containing the identity, when $c_1(E)$ belongs to the kernel of $\tau^*$, the transgression map on cohomology.  In Section \ref{ReducedCircleAction} we show that the reduced bundle, $L^0FrE$, admits an $S^1$-action such that the inclusion $L^0FrE \hookrightarrow LFrE$ is an $S^1$-equivariant map, inducing  a map $H_{S^1}^*(LFrE) \to H^*_{S^1}(L^0FrE)$.  After restricting to this sub-bundle, we see that $S^1$-equivariant Chern-Weil techniques define an $S^1$-equivariant first Chern class, $c_1^{S^1}(\mathcal{E})$.  Moreover, Theorem \ref{FCCExtension} proves that $c_1^{S^1}(\mathcal{E})$ extends $c_1(E)$, and we present a criterion that detects when $c_1^{S^1}(\mathcal{E})$ is non-trivial.  Section \ref{NontrivialExamples} identifies a collection of loop spaces that admit non-trivial $c_1^{S^1}(\mathcal{E})$.

As a result, $S^1$-equivariant Chern-Weil techniques only partially extend to pushdown bundles over loop space.  It is a challenging problem to construct $S^1$-equivariant characteristic classes on loop space in general.  The author is unaware of a topological construction and it seems difficult to construct other characteristic classes by Bismut's modified $S^1$-equivariant Chern-Weil technique.  This is an interesting problem for future work.

We would like to thank Steven Rosenberg and Mahmoud Zeinalian for many helpful discussions.  We would like to thank David Fried for his suggestion that led to Section \ref{NontrivialExamples}.  We also would like to thank the referee for many helpful suggestions, especially regarding the various maps between equivariant cohomology theories in Section \ref{EquivariantCohomology}.

\section{$S^1$-equivariant cohomology and characteristic classes}\label{EquivariantCohomology}

This section gathers some basic definitions and important properties of various $S^1$-equivariant cohomology theories as a background for Sections \ref{LoopSpaceConstructions} and \ref{EquFCC}, where we discuss $S^1$-equivariant characteristic classes on loop space. In Section \ref{CartanModel} we define the $S^1$-equivariant cohomology of a $S^1$-manifold $N$, written $H_{S^1}^*(N)$, and the Cartan model for $S^1$-equivariant cohomology.  Sections \ref{PeriodicCohomology} and \ref{SuperCohomology} introduce completed periodic $S^1$-equivariant cohomology, $h_{S^1}^*(N)$, and super $S^1$-equivariant cohomology, $\bar{h}_{S^1}^*(N)$.  Section \ref{LocalizationThms} states the localization theorem for these $S^1$-equivariant cohomology theories, which says that these $S^1$-equivariant cohomology theories are determined on the fixed-point set of the $S^1$-action.

We describe two approaches to $S^1$-equivariant Chern-Weil theory on a finite rank bundle $E$ over a finite dimensional manifold $N$.  These techniques construct $S^1$-equivariant characteristic classes belonging to $H_{S^1}^*(N)$.  Section \ref{GVectorBundle} constructs $S^1$-equivariant characteristic classes by $S^1$-equivariant vector bundles and Section \ref{GPrincipalBundle} constructs the same classes by $S^1$-equivariant principal $U(n)$-bundles.  Along the way we identify the main points of the theory that differ from the loop space case discussed in Section \ref{LoopSpaceConstructions}.

\subsection{The Cartan model of $G$-equivariant cohomology} \label{CartanModel}

Throughout this paper we work with de Rham cohomology of complex-valued forms.  Let $G$ be a compact and connected Lie group and suppose $G$ acts on a manifold $N$. The \emph{$G$-equivariant cohomology} of $N$ is $H_G^*(N) \overset{\rm def}{=} H^*(N \times_G EG)$.   The Cartan model is a differential graded algebra that serves as an algebraic model for equivariant cohomology, often proving convenient for computations.  We summarize its construction below.

Consider the space $S(\mathfrak{g}^*) \otimes \Lambda^*(N)$.  By identifying the symmetric algebra $S(\mathfrak{g}^*)$ with the ring of polynomials on $\mathfrak{g}$, we may identify $S(\mathfrak{g}^*) \otimes \Lambda^*(N)$ with the ring of polynomial functions on $\mathfrak{g}$ valued in $\Lambda^*(N)$.  $G$ acts on this ring by
\begin{align*}
g \cdot (p \otimes \omega)(X) &= p({\rm Ad}_{g^{-1}}X) \otimes g^*\omega.
\end{align*}
This space is graded by declaring $\deg(p \otimes \omega) = 2\deg(p) + \deg(\omega)$.  Let $C_G(N) = (S(\mathfrak{g}^*) \otimes \Lambda^*(N) )^G$ denote the subspace invariant under this action.  We call an element of $C_G(N)$ an \emph{equivariant differential form}.  Take a basis $\{ X_j \}$ for $\mathfrak{g}$ and let $\{ u_j \}$ be the corresponding generators of $S(\mathfrak{g}^*)$ induced by the dual basis of $\mathfrak{g}^*$.  The \emph{equivariant differential} is $d_G = d - \sum_j u_j \otimes i_{X_j}$, which acts on this complex by
\begin{align*}
d_G(p \otimes \omega)  &= p \otimes d\omega - \sum_j u_j p \otimes i_{X_j} \omega.
\end{align*}
Note that this definition does not depend on the choice of basis $\{ X_j \}$.  This operator is a differential because $d_G^2 = - \sum_j u_j \otimes \mathcal{L}_{X_j}$, which is the zero operator on invariant elements.  Notice that if $X = c^k X_k \in \mathfrak{g}$,
\begin{align*}
d_G(p \otimes w)(X) &= (p \otimes d\omega)(X) - \sum_j (u_j p \otimes i_{X_j} \omega)(X)\\
&= p(X) \otimes d\omega - \sum_j u_j(c^kX_k)p(X) \otimes i_{X_j} \omega\\
&= p(X) \otimes d\omega - \sum_j c^j p(X)\otimes i_{X_j} \omega\\
&= p(X) \otimes d\omega - p(X) \otimes i_X\omega
\end{align*}
For this reason some authors, such as \cite{BGV}, write the equivariant differential as $d - i_X$.  The complex $(C_G(N), d_G)$ is called the \emph{Cartan model} of equivariant cohomology and its cohomology is isomorphic to $H_G^*(N)$ \cite{GS}.

Consider the case $G = S^1$.  Because $S^1$ is abelian, its adjoint action is trivial and $(S(\mathfrak{g}^*) \otimes \Lambda^*(N) )^G = S(\mathfrak{g}^*) \otimes \Lambda_{S^1}^*(N)$, where $\Lambda_{S^1}^*(N)$ is the subspace of differential forms on $N$ invariant under the $S^1$-action.  Furthermore $S(\mathfrak{g}^*) = \mathbb{R}[u]$ for some generator $u \in \mathfrak{g}^*$ dual to the generator $X \in \mathfrak{g}$, so we may write the Cartan model as $(\Lambda_{S^1}^*(N)[u], d - ui_X)$, omitting the tensor product for convenience. 

\subsection{Periodic $S^1$-equivariant cohomology} \label{PeriodicCohomology}

Important theorems in the study of equivariant cohomology, like the localization formula of \cite[\textsection 7.2]{BGV}, involve rational maps $\mathfrak{g} \to \Lambda^*(N)$ rather than polynomial maps. For this reason we consider a variation of the Cartan model that includes such rational maps.  Consider the case $G = S^1$ and the Cartan model $(\Lambda_{S^1}^*(N)[u], d - ui_X)$.    The \emph{periodic $S^1$-equivariant cohomology} of $N$ is the cohomology of $(\Lambda_{S^1}^*(N)[u, u^{-1}], d - ui_X)$.  It is the localization of equivariant cohomology and it is denoted in \cite{JP} by $u^{-1}H_{S^1}^*(N)$.  This cohomology theory is called \emph{periodic} because of
\begin{proposition}
For any $k$, $u^{-1}H_{S^1}^k(N) \cong u^{-1}H_{S^1}^{k + 2}(N)$.
\end{proposition}
\begin{proof}
We claim that the map $T : u^{-1}H_{S^1}^k(N) \to u^{-1}H_{S^1}^{k + 2}(N)$ given by $T[\omega] = [u \omega]$ is an isomorphism.  It is straightforward to check that $u(d - ui_X) = (d - ui_X)u$ as operators on this complex.  It follows that $u\omega$ is equivariantly closed if and only if $\omega$ is equivariantly closed and $u \omega$ is equivariantly exact if and only if $\omega$ is equivariantly exact.  Therefore $T$ is a well defined linear map whose inverse is given by $T^{-1}[\omega] = [u^{-1}\omega]$.
\end{proof}

Another variant on the Cartan model is \emph{completed periodic $S^1$-equivariant cohomology}, written $h_{S^1}^*(N)$, which is the cohomology of $(\Lambda_{S^1}^*(N)[u, u^{-1}]], d - ui_X)$.  Notice that completed periodic $S^1$-equivariant cohomology enjoys the same periodicity property $h_{S^1}^k(N) \cong h_{S^1}^{k + 2}(N)$, where the isomorphism is again given by $[\omega] \mapsto [u \omega]$.  If $N$ is finite dimensional, $u^{-1}H_{S^1}^*(N)$ and $h_{S^1}^*(N)$ are isomorphic, though the two cohomology theories may differ for infinite dimensional manifolds.  Moreover, $u^{-1}H_{S^1}^*(N)$ is trivial when $N$ is infinite dimensional, while $h_{S^1}^*(N)$ need not be.  See \cite[\textsection 1]{JP} for a more detailed discussion of $u^{-1}H_{S^1}^*(N)$ and $h_{S^1}^*(N)$.

Let $N_0$ be the fixed point set of the $S^1$-action.  Assume that $N_0$ has an $S^1$-invariant neighborhood $U$ such that the inclusion $i : N_0 \to U$ is an $S^1$-equivariant homotopy equivalence.  An $S^1$-manifold that admits such a neighborhood is called \emph{regular}.  In \cite{JP}, Jones and Petrack prove
\begin{theorem}\label{PeriodicRestriction} If $N$ is a regular $S^1$-manifold, then the inclusion of the fixed point set $i : N_0 \to N$ induces an isomorphism
\begin{align*}
i^* : h_{S^1}^*(N) \cong h_{S^1}^*(N_0).
\end{align*}
\end{theorem}
Finite dimensional manifolds are regular, as is the loop space $LM$ of a finite dimensional manifold $M$.  However, not all infinite dimensional $S^1$-manifolds are regular \cite{JP}.

For our purposes we only consider periodic $S^1$-equivariant cohomology and completed periodic $S^1$-equivariant cohomology, though these cohomology theories can be defined for any torus.  A presentation of the general case can be found in \cite[Ch. 10]{GS}.

\subsection{Super $S^1$-equivariant cohomology} \label{SuperCohomology}

Although we cannot directly compare cohomology classes in $H_{S^1}^*(N)$ and $h_{S^1}^*(N)$, we may compare them in a third $S^1$-equivariant cohomology theory $\bar{h}_{S^1}^*(N)$, defined below.

Given a real parameter $s$, we may take the quotients $\Lambda_{S^1}^*(N)[u, u^{-1}]]/(u - s)$ and $\Lambda_{S^1}^*(N)[u]/(u - s)$.  This has the effect of setting $u = s$.  In \cite[\textsection 5]{AB}, Atiyah and Bott prove 
\begin{align}\label{ABLemma}
H^*(\Lambda_{S^1}^*(N)[u]/(u - s), d - si_X) &\cong H^*( \Lambda^*_{S^1}(N)[u], d - ui_X)/(u - s),\\
H^*(\Lambda_{S^1}^*(N)[u, u^{-1}]]/(u - s), d - si_X) &\cong H^*( \Lambda^*_{S^1}(N)[u, u^{-1}]], d - ui_X)/(u - s), \nonumber
\end{align}

so that we may interchange setting $u = s$ and taking the cohomology of our complex.  We may represent elements of these cohomology groups with differential forms because of

\begin{proposition}\label{EvaluateU}
Every element of $\Lambda^*_{S^1}(N)[u]/(u + 1)$ has a unique representative that is purely a differential form (an element in $\Lambda_{S^1}^*(N)[u]$ that depends on $u$ as a degree zero polynomial).  The same holds for every element of $\Lambda_{S^1}^*(N)[u, u^{-1}]]/(u + 1)$.
\end{proposition}

\begin{proof}
Let $\sum_k u^k \alpha_k + (u + 1) \in \Lambda^*_{S^1}(N)[u]/(u + 1)$,  and consider $\sum_k ( -1)^k \alpha_k + (u + 1)$.  Then
\begin{align*}
\sum_k u^k \alpha_k - \sum_k (-1)^k  \alpha_k &= \sum_k (u^k - (-1)^k) \alpha_k\\
&= \sum_k (u + 1)(u^{k - 1} - u^{k - 2} + \ldots + \pm 1)\alpha_k\\
&= (u + 1) \sum_k (u^{k - 1} - u^{k - 2} +  \ldots + \pm 1) \alpha_k,
\end{align*}
which implies that $\sum_k u^k \alpha_k + (u + 1) = \sum_k (-1)^k \alpha_k + (u + 1)$ in $\Lambda_{S^1}^*(N)[u]/(u + 1)$.  By associating the representative $\sum_k (-1)^k \alpha_k$ with the differential form $\sum_k (-1)^k \alpha_k$, we see that we can represent every element of $\Lambda_{S^1}^*(N)[u]/( u + 1)$ by a pure differential form.

Moreover, this representation is unique.  For if $a$ and $b$ are two such representatives, then $a - b \in (u + 1)$.  On the other hand, $a - b \in \Lambda_{S^1}^*(N)$.  Thus $a - b \in (u + 1) \cap \Lambda_{S^1}^*(N) = \{ 0 \}$, proving that $a = b$.

The proof goes through without change in the case of $\Lambda_{S^1}^*(N)[u, u^{-1}]]$.
\end{proof}

There is another $S^1$-equivariant cohomology theory described in \cite{Bis}, \cite{JP}, and \cite{Witten} which we call \emph{super $S^1$-equivariant cohomology}.

\begin{mydef}
Super $S^1$-equivariant cohomology, written $\bar{h}_{S^1}^*(N)$, is the cohomology of the complex $(\Lambda^*_{S^1}(N), d + i_X)$.  It has a $\mathbb{Z}/2\mathbb{Z}$ grading given by the parity of forms.
\end{mydef}

Both $H_{S^1}^*(N)$ and $h_{S^1}^*(N)$ can be mapped into $\bar{h}_{S^1}^*$ by setting $u = -1$.  Let $\{ \Lambda_{S^1}^*(N)[u] \}_{k}$ and $\{ \Lambda_{S^1}^*(N)[u, u^{-1}]] \}_{k}$ denote the degree $k$ subspaces of $\Lambda_{S^1}^*(N)[u]$ and $\Lambda_{S^1}^*(N)[u, u^{-1}]]$ respectively and let
\begin{align*}
&p_k : \{ \Lambda^*_{S^1}(N)[u]\}_{k} \to  \Lambda^*_{S^1}(N)[u]/(u + 1)\\
&q_k : \{ \Lambda^*_{S^1}(N)[u, u^{-1}]]\}_{k} \to \Lambda^*_{S^1}(N)[u, u^{-1}]]/(u + 1)
\end{align*}
denote the two quotient maps.  Interestingly, $q_k$ does not discard information, in the following sense.
\begin{proposition} \label{QIsomorphism}
For any $k$, the maps
\begin{align*}
&q_{2k} : h_{S^1}^{2k}(N) \to \bar{h}_{S^1}^{\rm even}(N)\\
&q_{2k + 1} : h_{S^1}^{2k + 1}(N) \to \bar{h}_{S^1}^{\rm odd}(N)
\end{align*}
are isomorphisms.
\end{proposition}
\begin{proof}
We define a map $r_{2k}$ inverse to $q_{2k} : \{ \Lambda_{S^1}^*[u, u^{-1}]] \}_{2k} \to \Lambda^{\rm even}_{S^1}(N)$.  Recall that $\Lambda_{S^1}^{\rm even}(N) = \prod_{n = 0}^\infty \Lambda_{S^1}^{2n}(N)$, so that an arbitrary $\omega \in \Lambda_{S^1}^{even}(N)$ can be written $\omega = \sum_{n = 0}^\infty \omega_{2n}$.  Set $r_{2k}(\omega) = \sum_{n = 0}^\infty u^{k - n}\omega_{2n}$.  In particular, $\deg r_{2k}(\omega) = 2k$, so $r_{2k} : \Lambda_{S^1}^{\rm even}(N) \to \{ \Lambda_{S^1}^*(N)[u, u^{-1}]]\}_{2k}$.  Similarly, any $\theta \in \Lambda_{S^1}^{\rm odd}(N)$ can be written $\theta = \sum_{n = 0}^\infty \theta_{2n + 1}$.  Define $r_{2k + 1} : \Lambda_{S^1}^{\rm odd}(N) \to \{ \Lambda_{S^1}^*(N)[u, u^{-1}]] \}_{2k + 1}$ by $r_{2k + 1}(\theta) = \sum_{n = 0}^\infty u^{k - n }\theta_{2n +1}$.

Given $\sum_{n = 0}^\infty u^{k - n}\omega_{2n} \in \{ \Lambda_{S^1}^*(N)[u, u^{-1}]] \}_{2k}$, we compute
\begin{align*}
 r_{2k}q_{2k} (\sum_{n = 0}^\infty u^{k - n}\omega_{2n}) &= r_{2k}(\sum_{n = 0}^\infty \omega_{2n}) = \sum_{n = 0}^{\infty} u^{k - n}\omega_{2n}
\end{align*}
Similarly, given $\sum_{n = 0}^\infty \omega_{2n} \in \Lambda_{S^1}^{\rm even}(N)$, we compute
\begin{align*}
q_{2k}r_{2k}( \sum_{n = 0}^\infty \omega_{2n}) = q_{2k}( \sum_{n = 0}^\infty u^{k - n}\omega_{2n} ) = \sum_{n = 0}^\infty \omega_{2n}
\end{align*}
Therefore $r_{2k}$ and $q_{2k}$ are inverses.  A similar computation shows $r_{2k + 1}$ and $q_{2k + 1}$ are inverses.

It is straightforward to check that $(d - ui_X)r_{2k} = r_{2k + 1}(d + i_X)$ and $(d + i_X)q_{2k} = q_{2k + 1}(d - ui_X)$, so $r_{2k}$ and $q_{2k}$ descend to isomorphisms on cohomology, and similarly for $r_{2k +1 }$ and $q_{2k + 1}$.
\end{proof}

\subsection{Localization on the fixed-point set} \label{LocalizationThms}

This section is adapted from \cite[\textsection 2]{JP}.  Combining Theorem \ref{PeriodicRestriction} and Proposition \ref{QIsomorphism} we have

\begin{theorem}\label{Localization}
If $N$ is a regular $S^1$-manifold, then the inclusion of the fixed point set $i : N_0 \to N$ induces an isomorphism
\begin{align*}
i^* : \bar{h}_{S^1}^*(N) \cong \bar{h}_{S^1}^*(N_0)
\end{align*}
\end{theorem}

\begin{remark}\label{LocalizationRemark}
Note that on $N_0$, $d + i_X = d$, the de Rham differential, because $X$ vanishes on $N_0$.  Moreover, $\Lambda_{S^1}^*(N_0) = \Lambda^*(N_0)$, as the circle action is trivial.  Therefore $\bar{h}_{S^1}^*(N_0) = H^{even/odd}(N)$, the de Rham cohomology of $N_0$, $\mathbb{Z}/2\mathbb{Z}$-graded by parity of forms.
\end{remark}

In particular, suppose $N = LM$ with the $S^1$-action given by rotation of loops. Then $N_0 = M$, embedded as the subspace of constant loops.

\begin{corollary}
$i^* : \bar{h}_{S^1}^{*}(LM) \cong \bar{h}_{S^1}^*(M) \cong H^{\rm even/odd}(M)$.
\end{corollary}

Thus a cohomology class in $\bar{h}_{S^1}^*(LM)$ is determined by its restriction to the embedding $M \hookrightarrow LM$.  We summarize these cohomology theories and the maps between them in the following table and diagram.
\bigskip

\begin{center}
 \begin{tabular}{| l | c | c | c | }
   \hline
 &   $H_{S^1}^*(N)$ & $h_{S^1}^*(N)$ & $\bar{h}_{S^1}^*(N)$ \\ \hline
Equivariant Forms &  $\Lambda_{S^1}^*(N)[u]$ & $\Lambda_{S^1}^*(N)[u, u^{-1}]]$ & $\Lambda_{S^1}^*(N)$ \\ \hline
Differential &   $d- ui_X$ & $d - ui_X$ & $d + i_X$ \\
   \hline
Localization &    - & $h_{S^1}^*(N) \cong h_{S^1}^*(N_0)$ & $\bar{h}_{S^1}^*(N) \cong \bar{h}_{S^1}^*(N_0)$ \\
   \hline
References & \cite{BGV, GS} & \cite{GJP, JP} & \cite{Atiyah, AB, Bis}\\
 \hline
 \end{tabular}
\end{center}
\bigskip

\begin{center}
\begin{tikzcd}
H_{S^1}^k(N)
\arrow{rd}{p_k}
&\hphantom{B} 
&h_{S^1}^k(N) \arrow{ld}[swap]{q_k}\\
&\bar{h}_{S^1}^{\rm even/odd}(N)
\end{tikzcd}
\end{center}

\subsection{$S^1$-equivariant vector bundles} \label{GVectorBundle}

 Let $\pi : E \to N$ be a rank $n$ complex vector bundle and suppose that $S^1$ acts on $E$ and $N$ by $\hat{k}_\theta$ and $k_\theta$ respectively such that $\pi \circ \hat{k}_\theta = k_\theta \circ \pi$ and $\hat{k}_\theta$ acts by vector bundle automorphisms.
\begin{align*}
\begin{CD}
E @>\hat{k}_\theta>> E\\
@VV \pi V @VV \pi V\\
N @>k_\theta>> N
\end{CD}
\end{align*}
We call $E \to N$ a \emph{$S^1$-equivariant vector bundle}.  $S^1$ acts on $\Gamma(E \to M)$ by
\begin{align*}
(k_\theta^{\Gamma}s)(x) \overset{\rm def}{=} \hat{k}_\theta s(k_{-\theta}x),
\end{align*}
where $s \in \Gamma(E \to N)$ and $x \in N$.  We say a connection $\nabla$ on $E$ is $S^1$-invariant if $\hat{k}_\theta^\Gamma \nabla = \nabla \hat{k}_\theta^\Gamma$.  Following \cite[Ch. 1]{BGV}, we may average a given connection $\nabla$ by the $S^1$-action,
\begin{align*}
\nabla^{ave} = \int_S^1 (k_{-\theta})^{\Gamma \otimes T^*M} \, \nabla \, k_\theta^\Gamma \, d\theta.
\end{align*}
That is, if $Y \in T_xN$,
\begin{align*}
\nabla^{ave}_Ys &= \int_{S^1} (k_{-\theta})^\Gamma \, \nabla_{{k_\theta}_*Y} \,  k_\theta^\Gamma s \, d\theta.
\end{align*}
It is straightforward to check that $k_\theta^\Gamma \nabla^{ave} = \nabla^{ave} k_\theta^\Gamma$, so $\nabla^{ave}$ is $S^1$-invariant.  Thus we may assume without loss of generality that our connection is $S^1$-invariant.

With our $S^1$-invariant connection $\nabla^{ave}$ we can define the $S^1$-equivariant curvature 2-form, an extension of the ordinary curvature 2-form to the Cartan model.  Let $X$ be the vector field on $N$ induced by the circle action.  We have the interior multiplication operator $i_X$ and the Lie derivative $\mathcal{L}_X$, related by Cartan's formula $\mathcal{L}_X = di_X + i_Xd$.  We first define the \emph{$S^1$-equivariant connection}, $\nabla^{S^1} \overset{\rm def}{=} \nabla^{ave} - ui_{X}$.  The \emph{$S^1$-equivariant curvature 2-form} is 
\begin{align*}
\Omega^{S^1} &\overset{\rm def}{=} (\nabla^{S^1})^2 + u\mathcal{L}_{X}.
\end{align*}
It is shown in \cite[Ch. 7]{BGV} that $\Omega^{S^1}$ belongs to $\Lambda_{S^1}^*(N, {\rm End}(E))[u]$, the Cartan model of forms on $N$ valued in ${\rm End}(E)$, and it has equivariant degree 2.  With the $S^1$-equivariant curvature 2-form, the techniques of Chern-Weil theory extend to the equivariant set-up and can be used to define \emph{$S^1$-equivariant characteristic classes}.  In Section \ref{LoopSpaceConstructions} we will be primarily interested in the \emph{$S^1$-equivariant Chern character},
\begin{align*}
{\rm ch}^{S^1}(E) &\overset{\rm def}{=} {\rm Tr}\exp \Omega^{S^1},
\end{align*}
an equivariantly closed form of mixed even degree.

\subsection{$S^1$-equivariant principal bundles} \label{GPrincipalBundle}

An alternative construction of $S^1$-equivariant characteristic classes uses $S^1$-equivariant principal bundles.  Throughout this section we follow \cite{BTChar}.  A principal $U(n)$-bundle $P$ is $S^1$-equivariant if $S^1$ acts on $P$ and $N$ on the left such that the projection $\pi : P \to N$ is $S^1$-equivariant and the left $S^1$-action commutes with the right $U(n)$-action.  Suppose we have the same set-up of a $S^1$-equivariant vector bundle $E \to N$ as in Section \ref{GVectorBundle}.  Without loss of generality we may assume that $\hat{k}_\theta : E \to E$ is a unitary transformation (for an arbitary metric on $E$ can be averaged by the $S^1$-action to produce a $S^1$-invariant metric).  Then $\hat{k}_\theta$ defines an action on the unitary frame bundle $FrE$,
\begin{align*}
\tilde{k}_\theta(x, e_1, \ldots, e_n) = (k_\theta x, \hat{k}_\theta e_1, \ldots, \hat{k}_\theta e_n).
\end{align*}
Moreover, the left $S^1$-action and the right $U(n)$-action commute.  A crucial point in Section \ref{LoopSpaceConstructions} is that the structure group action and the $S^1$-action do not commute for the infinite rank bundles we consider, in contrast to this finite rank case.  For that reason we now prove that the actions commute.  If $\theta \in S^1$ and $a \in U(n)$,
\begin{align*}
R_a[ \tilde{k}_\theta (x, e_1, \ldots, e_n)] &= R_a(k_\theta x, \hat{k}_\theta e_1, \ldots, \hat{k}_\theta e_n) = (k_\theta x, (\hat{k}_\theta e_1)a, \ldots, (\hat{k}_\theta e_n)a)\\
&= (k_\theta x, (\hat{k}_\theta e_j)a_1^{j}, \ldots, (\hat{k}_\theta e_j)a_n^j) = (k_\theta x, \hat{k}_\theta (e_ja_1^j), \ldots, \hat{k}_\theta (e_j a_n^j) )\\
&= \tilde{k}_\theta (x, e_j a_1^j, \ldots, e_j a_n^j) = \tilde{k}_\theta [ R_a(x, e_1, \ldots, e_n) ],
\end{align*}
proving $R_a \circ \tilde{k}_\theta = \tilde{k}_\theta \circ R_a$.   Thus $FrE \to N$ is an $S^1$-equivariant principal $U(n)$-bundle.
\begin{align*}
\begin{CD}
FrE @>\tilde{k}_\theta>> FrE\\
@VV V @VV V\\
N @>k_\theta>> N
\end{CD}
\end{align*}

Let $\omega \in \Lambda^1(FrE, \mathfrak{u}(n)  )$ be a connection 1-form.  We will show that $\omega^{ave} \overset{\rm def}{=} \int_{S^1} (\tilde{k}_\theta^*\omega)\, d\theta$ is an $S^1$-invariant connection 1-form, following Lemmas \ref{FundamentalVF} and \ref{PullbackConnection}.  

\begin{lemma}\label{FundamentalVF}
Let $A \in \mathfrak{u}(n)$ and let $A^*$ be its fundamental vector field.  Then $ \tilde{k}_{\theta*}(A^*|_{p}) = A^*|_{\tilde{k}_\theta p}$, for $p \in FrE$.
\end{lemma}

\begin{proof}
\begin{align*}
\tilde{k}_{\theta *} (A^*|_p) &= \frac{d}{dt}\bigg|_{t = 0} \tilde{k}_\theta( p \exp(tA) ) = \frac{d}{dt}\bigg|_{t = 0} (\tilde{k}_\theta p)\exp(tA) = A^*|_{\tilde{k}_\theta p}.
\end{align*}
\end{proof}

\begin{lemma}\label{PullbackConnection}
$\tilde{k}_\theta^*\omega$ is a connection 1-form.
\end{lemma}

\begin{proof}
We must show 1.) $\tilde{k}_\theta^*\omega(A^*) = A$ for $A \in \mathfrak{u}(n)$, and 2.) $R_a^* (\tilde{k}_\theta^*\omega) = Ad_{a^{-1}}(\tilde{k}_\theta^*\omega)$ for $a \in U(n)$.  To verify 1.), we compute
\begin{align*}
(\tilde{k}_\theta^*\omega)_p(A^*) &= \omega_{\tilde{k}_\theta p}(\tilde{k}_{\theta *} (A^*|_p) ) = \omega_{\tilde{k}_\theta p}( A^*|_{\tilde{k}_\theta p} ) = A.
\end{align*}
To verify 2.), note that commutativity of the group actions implies ${R_a}_* \tilde{k}_{\theta *} = \tilde{k}_{\theta *} {R_a}_*$.  We compute
\begin{align*}
R_a^*(\tilde{k}_\theta^*\omega)(X) &= \omega(\tilde{k}_{\theta *} {R_a}_* X) = \omega( {R_a}_* \tilde{k}_{\theta *} X)\\
= R_a^*\omega( \tilde{k}_{\theta *}X) &= Ad_{a^{-1}} \omega(\tilde{k}_{\theta *}X) = Ad_{a^{-1}} (\tilde{k}_\theta^*\omega)(X), 
\end{align*}
proving 2.).  Therefore $\tilde{k}_\theta^*\omega$ is a connection 1-form.
\end{proof}

With these lemmas, the following proposition shows that we may average a connection to produce an $S^1$-invariant connection.  In contrast to this finite dimensional case, we see in Section \ref{LoopSpaceConstructions} that the same construction fails for the $S^1$-action on loop space.

\begin{proposition}\label{AverageConnection}
$\omega^{ave} = \int_{S^1} \tilde{k}_\theta^*\omega \, d\theta$ is a $S^1$-invariant connection 1-form.
\end{proposition}

\begin{proof}
First we must show that $\omega^{ave}$ is a connection 1-form.  We must check 1.) $\omega^{ave}(A^*) = A$ for $A \in \mathfrak{u}(n)$, and 2.) $R_a^* \omega^{ave} = Ad_{a^{-1}}\omega^{ave}$.  To verify 1.), we compute
\begin{align*}
\omega^{ave}(A^*) &= \int_{S^1} (\tilde{k}_\theta^*\omega)(A^*)\, d\theta = \int_{S^1} A \, d\theta = A.
\end{align*}
To verify 2.), we compute
\begin{align*}
R_a^* \omega^{ave}(Y) &= R_a^* \int_{S^1} \tilde{k}_\theta^*\omega(Y) d\theta = \int_{S^1} R_a^*(\tilde{k}_\theta^*\omega)(Y) d\theta\\
&= \int_{S^1} Ad_{a^{-1}}(\tilde{k}_\theta^*\omega)(Y) d\theta = Ad_{a^{-1}} \int_{S^1} \tilde{k}_\theta^*\omega(Y) d\theta = Ad_{a^{-1}} \omega^{ave}(Y).
\end{align*}
Therefore $\omega^{ave}$ is a connection 1-form.  Moreover, for a fixed $\theta_0 \in S^1$,
\begin{align*}
\tilde{k}_{\theta_0}^*\omega^{ave}(X) &= \int_{S^1} \tilde{k}_{\theta_0}^*\tilde{k}_\theta^*\omega(X) d\theta = \int_G \tilde{k}_{\theta'}^*\omega(X) d\theta' = \omega^{ave}(X),
\end{align*}
where $\theta' = \theta + \theta_0$, proving that $\omega^{ave}$ is $S^1$-invariant.
\end{proof}

The moral of the story is that given an $S^1$-equivariant vector bundle $E \to N$, one can produce an $S^1$-invariant connection 1-form on $FrE \to N$, so we may assume without loss of generality that $\omega$ is $S^1$-invariant.

The \emph{$S^1$-equivariant curvature 2-form} is
\begin{align*}
\Omega^{S^1} &\overset{\rm def}{=} \Omega - u\omega(X)
\end{align*}
where $\Omega$ is the (ordinary) curvature 2-form of the connection $\omega$ and $X$ is the vector field on $FrE$ induced by the $S^1$-action.  It is an equivariant extension of the ordinary curvature 2-form, belonging to $\Lambda_{S^1}^*(FrE, \mathfrak{u}(n))[u]$ with equivariant degree 2.  Letting $D$ denote the covariant exterior derivative associated to the connection $\omega$, we have the equivariant Bianchi identity
\begin{align*}
(D - ui_X) \Omega^{S^1} = 0,
\end{align*}
proven in \cite{BTChar}.  With the equivariant Bianchi identity it is straightforward to check that $f(\Omega^{S^1})$ is equivariantly closed for any $U(n)$-invariant polynomial $f$.  Recall that a form on $FrE$ is \emph{basic} if it is $U(n)$-invariant and horizontal.  Basic forms are precisely those forms that descend to the base, i.e. those forms $\alpha \in \Lambda^*(FrE)$ such that $\alpha = \pi^*\beta$ for some $\beta \in \Lambda^*(N)$.  In Section \ref{LoopSpaceConstructions} we see that the equivariant differential forms we define fail to be basic, so for comparison's sake we now recall the standard proof that $f(\Omega^{S^1})$ is basic.

\begin{proposition}\label{StandardBasic}
If $f$ is a $U(n)$-invariant polynomial, $f(\Omega^{S^1})$ is basic.
\end{proposition}

\begin{proof}
We must show that $f(\Omega^{S^1})$ is horitzontal and $U(n)$-invariant.  Suppose $v$ is a vertical vector.  Then $i_v \Omega = i_v \omega(X) = 0$, and
\begin{align*}
i_vf(\Omega^{S^1}) = f(i_v\Omega^{S^1}) = 0,
\end{align*}
proving that $f(\Omega^{S^1})$ is horizontal.  In \cite[\textsection 6]{BTChar}, Bott and Tu show that for $a \in U(n)$,
\begin{align}\label{EquivariantTransformationLaw}
R_a^*\omega(X) = {\rm Ad}_{a^{-1}}\omega(X),
\end{align}
from which it follows that $R_a^*\Omega^{S^1} = {\rm Ad}_{a^{-1}}\Omega^{S^1}$.  Thus
\begin{align}\label{FiniteDimInvariant}
R_a^*f(\Omega^{S^1}) = f(R_a^*\Omega^{S^1}) = f({\rm Ad}_{a^{-1}}\Omega^{S^1}) = f(\Omega^{S^1}),
\end{align}
proving that $f(\Omega^{S^1})$ is $U(n)$-invariant.
\end{proof}

In particular we see that ${\rm Tr}(\Omega^{S^1})^k \in \Lambda_{S^1}^*(FrE)[u]$ is equivariantly closed and basic, implying that ${\rm Tr}\exp \Omega^{S^1} = \pi^*\beta$ for some $\beta \in \Lambda_{S^1}^*(N)[u]$.  In fact, $\beta = {\rm ch}^{S^1}(E)$, as in Section \ref{GVectorBundle}.  In this way one constructs $S^1$-equivariant characteristic classes via $S^1$-equivariant principal $U(n)$-bundles.

\section{Loop space and pushdown bundles}\label{LoopSpaceConstructions}

In this section we attempt to construct $S^1$-equivariant characteristic classes on $LM$ by $S^1$-equivariant Chern-Weil techniques on pushdown bundles and their associated frame bundles.  In Section \ref{PushdownBundles} we introduce the pushdown bundle over $LM$, an $S^1$-equivariant vector bundle of infinite rank.  In Section \ref{CurvatureOperators} we try to apply the $S^1$-equivariant Chern-Weil techniques via a covariant derivative on the pushdown bundle as attempted in \cite[\textsection{3}]{LMRT}.  These techniques fail for subtle reasons not present in the finite dimensional case.  In Section \ref{PrincipalBundleApproach} we attempt to apply $S^1$-equivariant Chern-Weil techniques via the principal $LU(n)$-bundle $LFrE$, and in this set-up it becomes clearer why these techniques fail.

In Section \ref{BismutChernCharacter} we recall the construction of ${\rm BCh}$, a characteristic class on $LM$ constructed by a modification of techniques from Section \ref{GPrincipalBundle} and Section \ref{PrincipalBundleApproach}.  However, ${\rm BCh}$ does not define a class in $H_{S^1}^*(LM)$, but rather a class in $\bar{h}_{S^1}(LM)$. 

\subsection{Pushdown bundle basics}\label{PushdownBundles}

Let $\pi : E \to M$ be a rank $n$ complex vector bundle over $M$, a finite dimensional smooth manifold.  Suppose that $E$ comes equipped with a hermitian metric and a metric connection $\nabla$.  We consider the free loop space $LM = C^\infty(S^1, M)$ with the Fr\'{e}chet topology.  The evaluation map $ev : LM \times S^1 \to M$ is given by $ev(\gamma, \theta) = \gamma(\theta)$.  We form the pullback bundle $ev^* E \to LM \times S^1$, a rank $n$ bundle over $LM \times S^1$.  Letting $\pi_1 : LM \times S^1 \to LM$ denote the projection on the first factor, we form the \emph{pushdown bundle} ${\pi_1}_* ev^* E = \mathcal{E} \to LM$.
\begin{align*}
\begin{CD}
@. ev^*E @> ev^* >> E  \\
@. @VV V @VV V \\
@. LM \times S^1 @> ev >> M \\
@. @VV \pi V\\
\mathcal{E} = {\pi_1}_*ev^*E @> >> LM
\end{CD}
\end{align*}
The pushdown bundle is an infinite rank bundle over $LM$, with fiber
\begin{align*}
\mathcal{E}_{\gamma} = \Gamma(\gamma^*E \to S^1) = \{ s : S^1 \to E : s(\theta) \in E_{\gamma(\theta)} \},
\end{align*}
given the Fr\'{e}chet topology.  In fact, $\mathcal{E}$ is isomorphic to the bundle $\tilde{\pi} : LE \to LM$, where $LE$ is the loop space of $E$ and whose projection is given by $\tilde{\pi}(\gamma) = \pi \circ \gamma$.

A trivialization of $\mathcal{E}$ near $\gamma_0 \in LM$ may be obtained in the following way.  Let $U_\epsilon \subset LM$ be a neighborhood of $\gamma_0$ given by `short curves'
\begin{align*}
U_\epsilon = \{ \gamma : \gamma(\theta) = \exp_{\gamma_0(\theta)}X(\theta), |X(\theta)|<\epsilon \},
\end{align*}
where $\epsilon > 0$ is chosen to be less than the injectivity radius of $\exp_{\gamma(\theta)}$ for all $\theta \in S^1$.  Note
\begin{align*}
\mathcal{E}_{\gamma_0} \overset{{\rm def}}{=} \Gamma( \gamma_0^*E \to S^1) \cong \Gamma( S^1 \times \mathbb{C}^n \to S^1) \cong L\mathbb{C}^n
\end{align*}
where the first isomorphism follows from the fact that any rank $n$ complex vector bundle over $S^1$ is trivial.  Notice, however, that the isomorphism is not canonical as it depends on the trivialization of $\gamma_0^*E \to S^1$.  Suppose $\gamma \in U_\epsilon$.  Then $\gamma(\theta) = \exp_{\gamma_0(\theta)}X(\theta)$ for some vector field $X(\theta)$ along $\gamma_0$, and for each $\theta \in S^1$ there is a curve $c_\theta(t) : [0, 1] \to M$ given by $c_{\theta}(t) = \exp_{\gamma_0(\theta)}tX(\theta)$.  This curve begins at $c_{\theta}(0) = \gamma_0(\theta)$ and ends at $c_{\theta}(1) = \gamma(\theta)$.  Thus we have an isomorphism $\parallel_{c_\theta} : E_{\gamma(\theta)} \to E_{\gamma_0(\theta)}$ given by parallel translation `backwards' along $c_{\theta}$.

We define an isomorphism $T_{\gamma} : \mathcal{E}_{\gamma} \to \mathcal{E}_{\gamma_0}$ in the following way.  Given $s(\theta) \in \mathcal{E}_{\gamma}$, we set
\begin{align*}
T_{\gamma}s(\theta) = \parallel_{c_\theta} s(\theta).
 \end{align*}
Thus $T$ defines a local trivialization, $T : \mathcal{E}|_{U_\epsilon} \to U_\epsilon \times \mathcal{E}_{\gamma_0}$, given by $(\gamma, s) \mapsto (\gamma, T_{\gamma}s )$.

\begin{remark}
Because $\mathcal{E}_{\gamma_0} \cong L\mathbb{C}^n$ (non-canonically), we may define a trivialization $\mathcal{E}|_{U_\epsilon} \to U_\epsilon \times L\mathbb{C}^n$.
\end{remark}

Next we examine the transition functions between trivializations.  Suppose $\gamma_1 \in LM$ and $V_\delta \subset LM$ is a neighborhood of $\gamma_1$ given by `short curves.'  We may define another trivialization $\widehat{T} : \mathcal{E}|_{V_\delta} \to V_\delta \times \mathcal{E}_{\gamma_1}$ in the same way.  That is, given $\gamma \in V_\delta$, for each $\theta \in S^1$ there is a curve $d_\theta(t) = \exp_{\gamma_1(\theta)} tY(\theta)$ beginning at $d_\theta(0) = \gamma_1(\theta)$ and ending at $d_\theta(1) = \gamma(\theta)$.  Then we define $\widehat{T}_\gamma$ as above.  Suppose further that $U_\epsilon \cap V_\delta \neq \emptyset$ and we have the trivializations
\begin{align*}
U_\epsilon \cap V_\delta \times \mathcal{E}_{\gamma_1} \overset{\widehat{T}}{\longleftarrow} \mathcal{E}|_{U_\epsilon \cap V_\delta} \overset{T}{\to} U_\epsilon \cap V_\delta \times \mathcal{E}_{\gamma_0}.
\end{align*}
Consider the transition function $T \circ \widehat{T}^{-1}$.  If $s(\theta) \in \mathcal{E}_{\gamma_1}$, 
\begin{align*}
T \circ \widehat{T}^{-1}s(\theta) &= T \parallel_{d_\theta}^{-1} s(\theta) = \parallel_{c_\theta} \parallel_{d_\theta}^{-1}s(\theta).
\end{align*}
Notice that $\parallel_{c_\theta} \parallel_{d_\theta}^{-1} : E_{\gamma_1(\theta)} \to E_{\gamma_0(\theta)}$ is an isometry of finite dimensional vector spaces.

Suppose we take a frame $\{ e_i(\theta)\}$ for $\gamma_0^*E \to S^1$ and a frame $\{ \tilde{e}_i(\theta) \}$ for $\gamma_1^*E \to S^1$.  For each $\theta \in S^1$, $\parallel_{c_\theta} \parallel_{d_\theta}^{-1} \tilde{e}_i(\theta) = u_i^j(\theta) e_j(\theta)$ where $u_i^j(\theta) \in \mathbb{C}$.  Moreover, because $\parallel_{c_\theta} \parallel_{d_\theta}^{-1}$ is an isometry, $u(\theta) = (u_i^j(\theta)) \in U(n)$ for each $\theta$.  The parallel translation operators $\parallel_{c_\theta}$ and $\parallel_{d_\theta}$ depend smoothly on $\theta$, implying that $u(\theta) \in LU(n)$.  Using the isomorphism $\mathcal{E}_{\gamma_i} \cong L\mathbb{C}^n$ ($i = 0, 1$) determined by our local frames, we may write $T \circ \widehat{T}^{-1} : L\mathbb{C}^n \to L\mathbb{C}^n$,
\begin{align*}
T \circ \widehat{T}^{-1}(f^1(\theta), \ldots, f^n(\theta)) \mapsto (u^1_j(\theta)f^j(\theta), \ldots, u^n_j(\theta)f^j(\theta) ).
\end{align*}
Written in vector notation $f(\theta) = (f^1(\theta), \ldots, f^n(\theta))$,
\begin{align*}
T \circ \widehat{T}^{-1} f(\theta) &= u(\theta) f(\theta).
\end{align*}
In particular we have shown that $\mathcal{E}$ admits the structure group $LU(n)$.  Because the natural representation of $LU(n)$ on  $L\mathbb{C}^n$ is smooth with respect to the Fr\'{e}chet topology \cite[\textsection 2, Theorem 2.3.3]{Hamilton}, we see that $\mathcal{E}$ is a Fr\'{e}chet bundle.

\subsection{$S^1$-equivariant curvature operators}\label{CurvatureOperators}

In this section we construct the $S^1$-equivariant curvature operator via a covariant derivative on $\mathcal{E}$.  However, we see that the $S^1$-equivariant Chern-Weil techniques outlined in Section \ref{GVectorBundle} fail to define $S^1$-equivariant characteristic classes on $LM$.

There is a natural isomorphism $\psi : \Gamma(\mathcal{E} \to LM) \to \Gamma( ev^*E \to LM \times S^1)$ given by $\psi(s) = \tilde{s}$, where $s$ and $\tilde{s}$ are related by
\begin{align*}
\tilde{s}(\gamma, \theta) = s(\gamma)(\theta).
\end{align*}
Under this isomorphism, an operator $D$ on $\Gamma(\mathcal{E} \to LM)$ is associated to an operator $\tilde{D}$ on $\Gamma( ev^*E \to LM \times S^1)$.
\begin{align*}
\begin{CD}
\Gamma(\mathcal{E} \to LM) @> D >> \Gamma(\mathcal{E} \to LM)\\
@VV \psi V @VV \psi V\\
\Gamma( ev^*E \to LM \times S^1) @> \tilde{D} >> \Gamma(ev^*E \to LM \times S^1)
\end{CD}
\end{align*}
We call an operator $D$ on $\Gamma(\mathcal{E} \to LM)$ a \emph{pointwise endomorphism} if for all $\gamma \in LM$ there is a bundle endomorphism $D_\gamma \in {\rm End}(\gamma^*E \to S^1)$ such that $(Ds)(\gamma) = D_\gamma s(\gamma)$ for $s \in \Gamma( \mathcal{E} \to LM)$.

\begin{proposition}\label{PointwiseCriterion}
An operator $D$ on $\Gamma(\mathcal{E} \to LM)$ is a pointwise endomorphism if and only if its associated operator $\tilde{D}$ is linear over $C^\infty(LM \times S^1)$.
\end{proposition}

\begin{proof}
First suppose $\tilde{D}$ is linear over $C^\infty(LM \times S^1)$.  Let $f \in C^\infty(LM \times S^1)$ and $s \in \Gamma(\mathcal{E} \to LM)$.  Then
\begin{align*}
D(fs)(\gamma)(\theta) &= \tilde{D}(f \tilde{s})(\gamma, \theta) = f \tilde{D} \tilde{s}(\gamma, \theta) = f D s (\gamma)(\theta).
\end{align*}
Let $(\gamma_0, \theta_0) \in LM \times S^1$ be fixed.  We may take a sequence $f_n$ of smooth bump functions whose support shrinks to $(\gamma_0, \theta_0)$ to conclude that $D s (\gamma)(\theta) = D_\gamma(\theta)s(\gamma)(\theta)$, for some $D_\gamma(\theta) \in {\rm End} \, E_{\gamma(\theta)}$, showing that $D_\gamma \in {\rm End}(\gamma^*E \to S^1)$.

Conversely, suppose that $D$ is a pointwise endomorphism.  Then $D_\gamma \in {\rm End}(\gamma^*E \to S^1)$ for all $\gamma \in LM$, and if $f \in C^\infty(LM \times S^1)$ and $s \in \Gamma(\mathcal{E} \to LM)$,
\begin{align*}
D (fs)(\gamma)(\theta) &= D_\gamma(\theta) f(\gamma, \theta) s(\gamma)(\theta) = f(\gamma, \theta) D_\gamma(\theta) s(\gamma)(\theta)\\
&= f(\gamma, \theta) D s(\gamma, \theta) = f Ds (\gamma)(\theta).
\end{align*}
In particular, our calculation shows that $\tilde{D}( f\tilde{s}) = f \tilde{D} \tilde{s}$, proving our claim.
\end{proof}

If $D$ is a pointwise endomorphism, we may take its \emph{leading order trace}
\begin{align*}
{\rm Tr} \, D (\gamma) \overset{\rm def}{=} \int_{S^1} {\rm tr}\,  D_\gamma(\theta) d\theta.
\end{align*}
It is straightforward to check that ${\rm Tr}$ defines a trace on the collection of pointwise endomorphisms.

A connection $\nabla$ on $E$ induces a connection $\nabla^{\mathcal{E}}$ on $\mathcal{E} \to LM$, defined as follows.  Suppose $s \in \Gamma(\mathcal{E})$, $\gamma \in LM$, $\theta \in S^1$, and $Y \in T_{\gamma}LM$.  Then we define $\nabla^{\mathcal{E}}$ by the equation
\begin{align*}
(\nabla_Y^{\mathcal{E}} s)_\gamma(\theta) = [ (ev^*\nabla^E)_{(Y, 0)}\tilde{s}]_{(\gamma, \theta)}.
\end{align*}
Thus $\nabla^{\mathcal{E}}$ is the operator on $\Gamma(\mathcal{E} \to LM)$ associated to $ev^*\nabla$ on $\Gamma( ev^*E \to LM \times S^1)$.

$S^1$ acts on $LM$ by rotation of loops, $(k_{\theta_0} \gamma)(\theta) = \gamma(\theta + \theta_0)$, and it acts on $\mathcal{E}$ by rotation of sections, $(\hat{k}_{\theta_0}s)(\theta) = s(\theta + \theta_0) \in \mathcal{E}|_{k_{\theta_0}\gamma}$.  These actions are compatible,
\begin{align*}
\begin{CD}
\mathcal{E} @> \hat{k}_\theta >> \mathcal{E}   \\
@VV  V @VV V  \\
LM @> k_\theta >> LM
\end{CD}
\end{align*}
which means $\mathcal{E} \to LM$ is an $S^1$-equivariant bundle.  Thus we have an $S^1$-action on $\Gamma(\mathcal{E} \to LM)$,
\begin{align*}
(k_{\theta_0}^\Gamma s)_\gamma(\theta) &\overset{\rm def}{=} (\hat{k}_{\theta_0} s)(k_{-\theta_0}\gamma)(\theta) = s({k_{-\theta_0} \gamma})(\theta + \theta_0).
\end{align*}
We average $\nabla^{\mathcal{E}}$ by the $S^1$-action to produce an $S^1$-invariant connection,
\begin{align*}
\nabla^{ave} = \frac{1}{2\pi}\int_{S^1} k_{-\theta}^{\Gamma \otimes T^*LM} \,  \nabla^{\mathcal{E}} \, k^\Gamma_\theta \, d\theta,
\end{align*}
as in Section \ref{GVectorBundle}.  Let $X$ be the vector field on $LM$ induced by the $S^1$-action.  We define the $S^1$-equivariant connection $\nabla^{S^1} = \nabla^{ave} - ui_X$ on $\mathcal{E} \to LM$, and the $S^1$-equivariant curvature $(\nabla^{S^1})^2 + u\mathcal{L}_X = \Omega^{S^1} \in \Lambda_{S^1}^*(LM, {\rm End}(\mathcal{E}))[u]$.  If $\Omega^{S^1}$ takes values in pointwise endomorphisms we may follow the $S^1$-equivariant Chern-Weil techniques of Section \ref{GVectorBundle} and define the $S^1$-equivariant characteristic forms
\begin{align*}
{\rm Tr} (\Omega^{S^1})^k &= \frac{1}{2\pi} \int_{S^1} {\rm tr} (\Omega^{S^1}_{\gamma(\theta)})^k \, d\theta,
\end{align*}
and an $S^1$-equivariant Chern character ${\rm Tr}\exp \Omega^{S^1}$.  For this reason, we must determine whether $\Omega^{S^1}$ takes values in pointwise endomorphisms of $\Gamma(\mathcal{E} \to LM)$.  We may write $\Omega^{S^1} = \Omega^{S^1}_{[2]} + u \Omega_{[0]}^{S^1}$, where $\Omega^{S^1}_{[k]} \in \Lambda^k_{S^1}(LM, {\rm End} \, \mathcal{E})$ and $\Omega^{S^1}_{[0]} = - [i_{X}, \nabla] + \mathcal{L}_{X}$.  Here $[ \cdot, \cdot ]$ denotes the superbracket as in \cite{BGV}.  If $\alpha \in \Lambda^*(LM)$ and $s \in \Gamma( \mathcal{E} \to LM)$, 
\begin{align*}
(-[i_{X}, \nabla^{ave}] + \mathcal{L}_{X})(\alpha \otimes s) &= -i_{X} \nabla^{ave} (\alpha \otimes s) - \nabla^{ave} i_{X} (\alpha \otimes s) + \mathcal{L}_{X} (\alpha \otimes s)\\
&= (-i_{X}d\alpha) \otimes s - (-1)^{|\alpha|} (i_{X} \alpha) \wedge \nabla^{ave} s - \alpha \otimes \nabla^{ave}_{X} s\\
&\hphantom{mbv}-(di_{X} \alpha) \otimes s - (-1)^{|\alpha| - 1} (i_{X} \alpha) \wedge \nabla^{ave} s\\
&\hphantom{mbv}+ (\mathcal{L}_{X} \alpha) \otimes s + \alpha \otimes (\mathcal{L}_{X} s)\\
&= \alpha \otimes (\mathcal{L}_{X} - \nabla^{ave}_{X})s.
\end{align*}

Let $D = \mathcal{L}_{X} - \nabla^{ave}_{X}$.  By Proposition \ref{PointwiseCriterion}, $\Omega_{[0]}^{S^1}$ takes values in pointwise endomorphisms if and only if the associated operator $\tilde{D}$ is linear over $C^{\infty}(LM \times S^1)$.  Suppose $s \in \Gamma(\mathcal{E} \to LM)$ and $f \in C^\infty(LM \times S^1)$.  By definition,
\begin{align*}
\mathcal{L}_{X} (fs) &=  \frac{d}{d\theta}\bigg|_{ \theta = 0} k_{-\theta}^\Gamma(fs).
\end{align*}
The $S^1$-action on $\Gamma(\mathcal{E} \to LM)$ induces an $S^1$-action on $\Gamma(ev^*E \to LM \times S^1)$ and an $S^1$-action on $C^\infty(LM \times S^1)$ by
\begin{align*}
\tilde{k}_{\theta_0}^\Gamma \tilde{s}(\gamma, \theta) = \tilde{s}(k_{-\theta_0}\gamma, \theta + \theta_0), \hphantom{bn} k_{\theta_0} f (\gamma, \theta) = f(k_{-\theta_0}\gamma, \theta + \theta_0).
\end{align*}
These $S^1$-actions are determined by the requirement that $k_{\theta_0}^\Gamma (fs) = (k_{\theta_0} f)(\tilde{k}_{\theta_0}^\Gamma \tilde{s})$, as
\begin{align*}
k_{\theta_0}^\Gamma(fs)(\gamma)(\theta) &= f(k_{-\theta_0}\gamma, \theta + \theta_0)s(k_{-\theta_0}\gamma)(\theta + \theta_0)\\
&= f(k_{-\theta_0}\gamma, \theta + t)\tilde{s}(k_{-\theta_0}\gamma, \theta + \theta_0) = (k_{\theta_0}f)(\tilde{k}^\Gamma_{\theta_0} \tilde{s})(\gamma, \theta).
\end{align*}
With these $S^1$-actions we have
\begin{align*}
\mathcal{L}_{X} (fs) &= \frac{d}{d\theta}\bigg|_{\theta = 0} k_{-\theta}^{\Gamma}(fs) = \frac{d}{d\theta}\bigg|_{\theta = 0} (k_{-\theta} f)(\tilde{k}_{-\theta}^\Gamma \tilde{s})\\
&= \frac{d}{d\theta}\bigg|_{\theta = 0} f(k_{\theta}\gamma, \theta_0 - \theta)\tilde{s}(k_{t}\gamma, \theta_0 - \theta) = \left[ \left(X - \frac{\partial}{\partial \theta}\right) f \right]s + f \mathcal{L}_{X}s.
\end{align*}
Thus we have shown that the associated operator $\tilde{\mathcal{L}}_{X}$ satisfies
\begin{align*}
\tilde{\mathcal{L}}_{X}(f \tilde{s}) = \left[ \left( X - \frac{\partial}{\partial \theta} \right) f \right]\tilde{s} + f \tilde{\mathcal{L}}_{X}\tilde{s}.
\end{align*}
On the other hand, we now consider the operator associated to $\nabla^{ave}_{Y}$.

\begin{align}\label{AssociatedAverage}
\tilde{\nabla}^{ave}_Y (fs) &= \int_{S^1} k_{-\theta}^{\Gamma} \tilde{\nabla}^{\mathcal{E}}_{k_\theta Y} k_\theta^{\Gamma} (fs)\, d\theta \nonumber \\
&=  \int_{S^1} \tilde{k}_{-\theta}^{\Gamma} ev^*\nabla^E_{(k_\theta Y, 0)} (k_\theta f) (\tilde{k}_\theta^{\Gamma} \tilde{s})\, d\theta \nonumber \\
&= \int_{S^1} \tilde{k}_{-\theta}^{\Gamma} \bigg( \big[(k_\theta Y)(k_\theta f)\big] (\tilde{k}_\theta^\Gamma \tilde{s}) + (k_\theta f) ev^*\nabla^E_{(k_\theta Y, 0)} (\tilde{k}_\theta^\Gamma \tilde{s}) \bigg) \, d\theta\\
&= \int_{S^1} \tilde{k}_{-\theta}\big[ (k_\theta Y)(k_\theta f)\big] \tilde{k}_{-\theta}^\Gamma \tilde{k}_\theta^\Gamma \tilde{s} + (k_{-\theta} k_\theta f) \tilde{k}_{-\theta}^\Gamma ev^* \nabla^E_{(k_\theta Y, 0)} (\tilde{k}_\theta^\Gamma s) \, d\theta \nonumber \\
&= \left( \int_{S^1} k_{-\theta}\big[ (k_\theta Y)(k_\theta f)\big] \, d\theta \right) s + f \tilde{\nabla}^{ave}_Y s. \nonumber
\end{align}

If $Y \in T_\gamma LM$, $(k_\theta Y, 0) \in T_{k_\theta \gamma}LM \oplus T_{\theta_0} S^1 = T_{(k_\theta \gamma, \theta_0) } LM \times S^1$.  Suppose $X = \frac{d}{d\epsilon} |_{\epsilon = 0} \gamma(\epsilon)$, for some curve $\gamma(\epsilon)$ with $\gamma(0)  = \gamma$.  Writing $(k_\theta Y, 0)$ as $k_\theta Y$, we have
\begin{align*}
(k_\theta Y)\bigg|_{(k_\theta \gamma, \theta_0)}(k_\theta f) &= \frac{d}{d\epsilon} \bigg|_{\epsilon = 0} (k_ \theta f)(k_\theta \gamma(\epsilon), \theta_0) = \frac{d}{d\epsilon}\bigg|_{\epsilon = 0} f(k_{-\theta}k_\theta \gamma(\epsilon), \theta_0 + \theta)\\
&= \frac{d}{d\epsilon}\bigg|_{\epsilon = 0} f(\gamma(\epsilon), \theta_0 + \theta) = Y\bigg|_{(\gamma, \theta_0 + \theta)}f.
\end{align*}

Therefore
\begin{align*}
k_{-\theta}[ (k_\theta Y) (k_\theta f)] \bigg|_{(\gamma, \theta_0)} &= (k_\theta Y)\bigg|_{(k_\theta \gamma, \theta_0 - \theta)} (k_\theta f) = Y\bigg|_{(\gamma, \theta_0)} f,
\end{align*}
and we have
\begin{align}\label{AverageSimplifies}
\int_{S^1} k_{-\theta}[ (k_\theta Y)(k_\theta f)]\, d\theta &= \int_{S^1} Yf  \, d\theta = Yf.
\end{align}
Combining equations (\ref{AssociatedAverage}) and (\ref{AverageSimplifies}), we have shown that
\begin{align*}
\tilde{\nabla}^{ave}_Y(f \tilde{s}) &= (Yf) \tilde{s} + f \tilde{\nabla}_Y^{ave} \tilde{s}.
\end{align*}
In particular, we have
\begin{align*}
\tilde{D} (f\tilde{s}) &= \tilde{\mathcal{L}}_{X}(f\tilde{s}) - \tilde{\nabla}^{ave}_{X}(f \tilde{s}) = \left[ \left( X - \frac{\partial}{\partial \theta} \right)f  \right]\tilde{s} + f \tilde{\mathcal{L}}_{X} \tilde{s} - (X f) \tilde{s} - f \tilde{\nabla}_{X}^{ave}\tilde{s}\\
&= - \left(\frac{\partial}{\partial \theta} f \right)\tilde{s} + f (\tilde{\mathcal{L}}_{X} - \tilde{\nabla}^{ave}_{X} )\tilde{s} = - \left( \frac{\partial}{\partial \theta} f \right) \tilde{s} + f \tilde{D} \tilde{s}.
\end{align*}
This computation shows that $\tilde{D}$ is not linear over $C^\infty(LM \times S^1)$, proving that $D$ is not valued in pointwise endomorphisms.  Thus we cannot take the leading order trace of $(\Omega^{S^1})^k$.  This approach was done incorrectly in \cite[\textsection 3]{LMRT}.

$S^1$-equivariant Chern-Weil theory requires a trace functional on $(\Omega^{S^1})^k$, for all $k \geq 0$.  Our calculations demonstrate that $\Omega^{S^1}$ is valued in first order operators, implying in this setting Chern-Weil techniques require a trace defined on some algebra of operators that includes differential operators of arbitrary order.  Our leading order trace is not defined on this algebra and so it does not define characteristic forms.  An alternative approach may be through the Wodzicki residue, ${\rm res}_W$, which is essentially the only trace on the full algebra of classical pseudodifferential operators.  However, ${\rm res}_W$ defines characteristic forms that do not extend the ordinary characteristic forms under the inclusion $M \hookrightarrow LM$.  (See \cite{PR} and \cite{MRTII} for a discussion of ${\rm res}_W$ and Wodzicki-Chern classes.)  Such $S^1$-equivariant Chern-Weil constructions using ${\rm res}_W$ may prove an interesting direction for future work.  

In summary, $S^1$-equivariant Chern-Weil techniques do not construct $S^1$-equivariant characteristic classes on $LM$ because $(\Omega^{S^1})^k$ takes values in an algebra of operators that does not admit a suitable trace, an analytic obstacle unique to infinite dimensional operators.

\subsection{Principal $LU(n)$-bundle approach}\label{PrincipalBundleApproach}

Let $\pi : FrE \to M$ be the unitary frame bundle of $E$.  It is a principal $U(n)$-bundle over $M$.  Moreover, it is shown in \cite[Theorem 4.6]{Stacey} that $\tilde{\pi} : LFrE \to LM$ admits the structure of a principal $LU(n)$-bundle, where $\tilde{\pi}(\gamma) = \pi \circ \gamma$.  Because $LU(n)$ is the structure group of $\mathcal{E}$, we may consider $LFrE$ to be the frame bundle of $\mathcal{E}$.  This idea is made precise in \cite[\textsection 4.3]{Stacey} and \cite[\textsection 2]{McLaughlin}.

The $S^1$-action on $LFrE$ covers the $S^1$-action on $LM$.
\begin{align*}
\begin{CD}
LFrE @> k_\theta >> LFrE\\
@VV V @VV V\\
LM @> k_\theta>> LM
\end{CD}
\end{align*}

The $S^1$-action on $LFrE$ induces a vector field on $LFrE$, which we also write $X$.  A connection 1-form $\omega \in \Lambda^1(FrE, \mathfrak{u}(n))$ on $FrE$ induces a connection 1-form $\tilde{\omega} \in \Lambda^1(LFrE, L\mathfrak{u}(n))$ on $LFrE$ given by
\begin{align*}
\tilde{\omega}(Y(\theta))\big|_{\theta = \theta_0} &= \omega(Y(\theta_0)),
\end{align*}
for $Y(\theta) \in T_\gamma LFrE$.  Notice that if $\theta_0 \in S^1$, $a \in LU(n)$, and $\gamma \in LFrE$, we have
\begin{align*}
[ R_a (k_{\theta_0}\gamma)](\theta) &= \gamma(\theta + \theta_0)  a(\theta),\\
[ k_{\theta_0}(R_a \gamma)](\theta) &= \gamma(\theta + \theta_0) a(\theta + \theta_0),
\end{align*}
proving $R_a \circ  k_{\theta_0} \neq k_{\theta_0} R_a$.

Thus the $S^1$-action and the $LU(n)$-action do not commute.  This is a significant departure from the finite dimensional case and we now explore its consequences.

\begin{lemma}\label{LoopFundamentalVF}
Let $A(\theta) \in L\mathfrak{u}(n)$ and let $A(\theta)^*$ be its fundamental vector field on $LFrE$.  For $\theta_0 \in S^1$, $k_{\theta_0 *} (A(\theta)^*|_{\gamma}) = A(\theta + \theta_0)|_{k_{\theta_0} \gamma}$.
\end{lemma}

\begin{proof}
\begin{align*}
k_{\theta_0 *} (A(\theta)^*|_{\gamma}) &= \frac{d}{dt} \bigg|_{t = 0} k_{\theta_0}( \gamma(\theta) \cdot \exp(t A(\theta)) ) = \frac{d}{dt}\bigg|_{t = 0} \gamma(\theta + \theta_0) \cdot \exp(t A(\theta + \theta_0) ) = A(\theta + \theta_0)^*|_{k_{\theta_0} \gamma}
\end{align*}
\end{proof}

One should compare this lemma to Lemma \ref{FundamentalVF}.  The difference arises because $S^1$ acts on $LU(n)$, a subtlety absent in the finite dimensional case.  A consequence of this difference is

\begin{proposition}
Let $\mu$ be a connection 1-form on $LFrE$.  Whenever $\theta_0 \in S^1$ is not the identity element, $k_{\theta_0}^* \mu$ is not a connection 1-form.
\end{proposition}

\begin{proof}
Suppose, to the contrary, that $k_{\theta_0}^*\mu \in \Lambda^1(LFrE, L\mathfrak{u}(n))$ is a connection 1-form.  Let $A(\theta) \in L\mathfrak{u}(n)$ and let $A(\theta)^*$ be its fundamental vector field on $LFrE$.  Because $k_{\theta_0}^*\mu$ is a connection 1-form,
\begin{align}\label{FirstEqn}
(k_{\theta_0}^*\mu)(A(\theta)^*) = A(\theta).
\end{align}
On the other hand, by Lemma \ref{LoopFundamentalVF},
\begin{align*}
(k_{\theta_0}^*\mu)(A(\theta)^*) = \mu(k_{\theta_0 *}A(\theta)^*) = \mu(A(\theta + \theta_0)^*) = A(\theta + \theta_0).
\end{align*}
contradicting equation (\ref{FirstEqn}).  Therefore $k_{\theta_0}^*\mu$ is not a connection 1-form.
\end{proof}

This proposition implies that the $S^1$-action on $LFrE$ does not induce an $S^1$-action on the space of connection 1-forms on $LFrE$, and so we cannot ask for an $S^1$-invariant connection 1-form.  Nevertheless, in analogy with Section \ref{GPrincipalBundle}, we consider $\tilde{\Omega} - u\tilde{\omega}(X)$, an equivariant 2-form valued in $L\mathfrak{u}(n)$.  Theorem 3.5 of \cite{Bis} states
\begin{align*}
D\tilde{\omega}(X) + i_{X}\tilde{\Omega} = 0,
\end{align*}
from which we see that
\begin{align*}
(D - ui_{X})(\tilde{\Omega} - u\tilde{\omega}(X)) &= D\tilde{\Omega} - uD\tilde{\omega}(X) -ui_{X}\tilde{\Omega} + u^2 i_{X}\tilde{\omega}(X) = 0,
\end{align*}
proving that the equivariant Bianchi identity holds.

\begin{theorem} \label{IndependentOfConnection}
${\rm Tr}(\tilde{\Omega} - u\tilde{\omega}(X))^k$ is equivariantly closed for all $k \geq 0$.  Moreover, the $S^1$-equivariant cohomology class $[{\rm Tr}(\tilde{\Omega} - u\tilde{\omega}(X))^k ] \in H_{S^1}^{2k}(LFrE)$ is independent of connection on $FrE$.
\end{theorem}

\begin{proof}
To prove the first claim, we note that
\begin{align*}
(d - ui_{X}){\rm Tr}(\tilde{\Omega} - u\tilde{\omega}(X) )^k &= {\rm Tr}(D - ui_{X})(\tilde{\Omega} - u\tilde{\omega}(X))^k = 0.
\end{align*}
To prove the second claim, suppose $\omega_0$ and $\omega_1$ are two connections on $FrE$ with induced connections $\tilde{\omega}_0$ and $\tilde{\omega}_1$.  Let $\alpha = \tilde{\omega}_1 - \tilde{\omega}_0$, $\tilde{\omega}_t = \tilde{\omega}_0 + t\tilde{\alpha}$, and let $\tilde{\Omega}_t$ be the curvature of $\tilde{\omega}_t$.  It is standard that $D_t\alpha = \frac{d}{dt} \tilde{\Omega}_t$ \cite[Ch. XII, Lemma 4]{KN}.  It follows that
\begin{align*}
(D_t - ui_{X})\alpha = \frac{d}{dt}(\tilde{\Omega}_t - u\tilde{\omega}_t(X)).
\end{align*}
Therefore
\begin{align*}
(d -ui_{X})k{\rm Tr}\big( \alpha \wedge (\tilde{\Omega}_t - u\tilde{\omega}_t(X))^{k - 1}\big) &= k {\rm Tr} (D_t - ui_{X})\big(\alpha \wedge (\tilde{\Omega}_t - u\tilde{\omega}_t(X))^{k - 1} \big) \\
&= k {\rm Tr} \big( ((D_t - ui_{X})\alpha) \wedge (\tilde{\Omega}_t - u\tilde{\omega}_t(X))^{k - 1} \big) \\
&= k {\rm Tr}\left(  \frac{d}{dt} (\tilde{\Omega}_t - u\tilde{\omega}_t(X)) \wedge (\tilde{\Omega}_t - u\tilde{\omega}_t(X))^{k - 1} \right) \\
&= \frac{d}{dt} {\rm Tr}(\tilde{\Omega}_t  - u\tilde{\omega}_t(X) )^k,
\end{align*}
and so
\begin{align*}
(d - ui_{X})& \int_0^1 k {\rm Tr}(\alpha \wedge (\tilde{\Omega}_t - u\tilde{\omega}_t(X))^{k - 1} dt = \int_0^1 \frac{d}{dt} {\rm Tr}(\tilde{\Omega}_t - u\tilde{\omega}_t(X) )^k dt\\
&= {\rm Tr}(\tilde{\Omega}_1 - u \tilde{\omega}_1(X) )^k  - {\rm Tr} (\tilde{\Omega}_0 - u \tilde{\omega}_0(X) )^k.
\end{align*}
\end{proof}

Thus we have a sequence of $S^1$-equivariant cohomology classes $[{\rm Tr}(\tilde{\Omega} - u\tilde{\omega}(X))^k] \in H_{S^1}^{2k}(LFrE)$.  We find, however, that these $S^1$-equivariant cohomology classes do not descend to $LM$, in contrast to the finite dimensional case presented in Section \ref{GPrincipalBundle}.  This follows from Theorem 3.5 of \cite{Bis}, which states that for $a \in LU(n)$,
\begin{align}\label{BismutsEqn3.6}
(R_a^* \omega(X))(\theta) = {\rm Ad}_{a(\theta)^{-1}}\omega(X) + a(\theta)^{-1}\dot{a}(\theta).
\end{align}
This formula differs from equation (\ref{EquivariantTransformationLaw}), which is the key property used to prove Proposition \ref{StandardBasic}.  This is an important difference between the finite dimensional and infinite dimensional cases.  As a result, we see that
\begin{align}\label{LUnTransgression}
R_a^* {\rm Tr }(\tilde{\Omega} - u\tilde{\omega}(X) ) &= {\rm Tr} \, (R_a^*\tilde{\Omega} - uR_a^*\tilde{\omega}(X)) \nonumber\\
&= {\rm Tr}( {\rm Ad}_{a^{-1}} \tilde{\Omega} - u {\rm Ad}_{a^{-1}}\tilde{\omega}(X) -u a^{-1}\dot{a}) \nonumber \\
&= {\rm Tr}(\tilde{\Omega} - u  \tilde{\omega}(X)) -u {\rm Tr}\, a^{-1}\dot{a} \nonumber\\
&= {\rm Tr}(\tilde{\Omega} - u \tilde{\omega}(X)) - u \frac{1}{2\pi} \int_{S^1} {\rm tr}\, a^{-1}\dot{a} \, d\theta\\
&= {\rm Tr}(\tilde{\Omega} - u \tilde{\omega}(X)) - u i W(\det a) \nonumber \\
&\neq {\rm Tr}\, (\tilde{\Omega} - u\tilde{\omega}(X)), \nonumber
\end{align}
where $W(\det a)$ is the winding number of $\det a : S^1 \to S^1$. The last equality follows because
\begin{align*}
 \int_{S^1} {\rm tr}\, a^{-1} \dot{a} \, d\theta &= \int_{S^1} {\rm tr}\, \frac{d}{dt} \log a \, d\theta = \int_{S^1} \frac{d}{dt} \log \det a \, d\theta = 2\pi i W(\det a).
\end{align*}

This calculation proves that ${\rm Tr}(\tilde{\Omega} -u\tilde{\omega}(X))$ is not $LU(n)$-invariant.  Therefore ${\rm Tr}(\tilde{\Omega} - u\tilde{\omega}(X))$ is not basic, i.e. ${\rm Tr}(\tilde{\Omega} - u\tilde{\omega}(X)) \neq \pi^* \beta$ for all $\beta \in \Lambda_{S^1}^*(LM)[u]$, and the equivariant differential forms ${\rm Tr}(\tilde{\Omega} - u\tilde{\omega}(X))^k$ do not define characteristic classes on $LM$.  Instead, we have an $S^1$-equivariant Chern character on $LFrE$, ${\rm ch}^{S^1}(\mathcal{E}) = {\rm Tr} \exp(\tilde{\Omega} - u \tilde{\omega}(X) )^k$, defining an $S^1$-equivariant characteristic class $[{\rm ch}^{S^1}(\mathcal{E})] \in H_{S^1}^*(LFrE)$.

The author is not aware of examples in which $[{\rm Tr}( \tilde{\Omega} - u\tilde{\omega}(X) )^k]$ is non-trivial.  Nevertheless, the following criterion may detect when these classes do not vanish.

\begin{proposition}
$[{\rm Tr}(\tilde{\Omega} - u \tilde{\omega}(X) ) ] \neq 0$ if $[{\rm tr}\, \Omega ] \neq 0$ in $H^2(FrE)$.
\end{proposition}

\begin{proof}
Let $i : M \to LM$ be the embedding of $M$ as the subspace of constant loops.  If $Y, Z \in T_pM$,
\begin{align*}
(i^*{\rm Tr}\, \tilde{\Omega})_p(Y, Z) &= {\rm Tr}\, \tilde{\Omega}_{i(p)}(i_*Y, i_*Z) = \frac{1}{2\pi} \int_{S^1} {\rm tr}\, \Omega_p(Y, Z) d\theta = {\rm tr}\, \Omega_p(Y, Z),
\end{align*}
proving $i^*{\rm Tr}\, \tilde{\Omega} = {\rm tr}\, \Omega$.  

It follows that if $[ {\rm tr} \, \Omega ] \neq 0$, then $[ {\rm Tr} \, \tilde{\Omega} ] \neq 0$ as well.  Therefore ${\rm Tr}(\tilde{\Omega} - u\tilde{\omega}(X) ) \neq (d - ui_X)\alpha$ for all $\alpha \in \Lambda_{S^1}^1(LFrE)$, as that would imply ${\rm Tr}\, \tilde{\Omega} = d\alpha$.
\end{proof}

In summary, this alternative construction via a principal $LU(n)$-bundle fails because the $S^1$-action does not commute with the $LU(n)$-action.  In contrast to the construction via a covariant derivative, this obstacle is not unique to infinite dimensions, since commutativity of the group actions is required for equivariant Chern-Weil techniques on finite dimensional manifolds, outlined in Section \ref{EquivariantCohomology}.

It is instructive to compare this construction with \cite[Remark 2]{Bis}, where Bismut reformulates this construction on a principal $LU(n) \rtimes S^1$-bundle over $LM$ for which the $S^1$-action commutes with the structure group action.  In this set-up, we must find a trace on the Lie algebra of $LU(n) \rtimes S^1$ to define characteristic forms.  The Lie algebra of $LU(n) \rtimes S^1$ is isomorphic to the collection of differential operators given by $b \frac{d}{d\theta} + A(\theta)$ acting on $C^\infty(S^1, \mathbb{C}^n)$, where $b \in \mathbb{R}$ and $A(\theta) \in \mathfrak{u}(n)$, so we find the same obstacle of requiring a suitable trace on differential operators acting on this infinite dimensional space.

\subsection{Bismut's construction}\label{BismutChernCharacter}

We end this section by presenting Bismut's construction of ${\rm BCh}$ \cite{Bis} to provide an example of how $S^1$-equivariant Chern-Weil techniques have been modified for the loop space setting.  Given $\gamma \in LFrE$, let $H(t)$ be the solution to the integral equation
\begin{align}\label{IntegralEqn}
H(t) = Id + \int_0^t H(v)(\Omega_{\gamma(v)}  + \omega(X) ) dv.
\end{align}
$H(t)$ is a power series of even-degree $S^1$-invariant differential forms valued in $L\mathfrak{u}(n)$.  In \cite[Theorem 3.7]{Bis} Bismut proves
\begin{align}\label{BismutsEqns}
(D + i_{X})H &= 0, \hphantom{bn} (R_a^* H)(t) = a(0)^{-1}H(t)a(t),
\end{align}
for $a \in LU(n)$.  The second equation in (\ref{BismutsEqns}) is the key to showing that ${\rm BCh}$ defines a class on $LM$, in contrast to the equivariant differential forms in Section \ref{PrincipalBundleApproach} defined by $S^1$-equivariant Chern-Weil techniques, so we prove it now.

Equation (\ref{BismutsEqn3.6}) says
\begin{align*}
R_a^*(\omega(X) + \Omega) = {\rm Ad}_{a^{-1}}(\omega(X) + \Omega) + a^{-1}\dot{a}.
\end{align*}
Combining this equation with (\ref{IntegralEqn}), we see that $R_a^*H$ satisfies the integral equation
\begin{align}\label{IntegralEquation}
(R_aH)(s) = Id + \int_0^s (R_aH)(v) {\rm Ad}_{a^{-1}} (\omega(X) + \Omega) + (R_aH)(v) a^{-1}(v) \dot{a}(v) dv.
\end{align}
Thus to prove the second equation in (\ref{BismutsEqns}), we need to check that $a^{-1}(0)H(s)a(s)$ satisfies equation (\ref{IntegralEquation}).  Using  (\ref{IntegralEqn}),
\begin{align*}
\frac{d}{ds} &\big[a^{-1}(0)H(s) a(s) \big]\\
&= a^{-1}(0) \left( \frac{d}{ds}H(s) \right)a(s) + a^{-1}(0)H(s) \dot{a}(s)\\
&= a^{-1}(0)H(s)(\omega(X) + \Omega) a(s) + a^{-1}(0)H(s)a(s)a^{-1}(s)\dot{a}(s)\\
&= \big[a^{-1}(0)H(s) a(s) \big]a^{-1}(s) (\omega(X) + \Omega) a(s) + \big[a^{-1}(0) H(s) a(s) \big] a^{-1}(s)\dot{a}(s).
\end{align*}
Moreover, 
\begin{align*}
a^{-1}(0)H(0)a(0) &= a^{-1}(0)( Id) a(0) = Id,
\end{align*}
proving that $a^{-1}(0)H(s)a(s)$ satisfies (\ref{IntegralEquation}). Therefore $(R_a^*H)(s) = a^{-1}(0)H(s)a(s)$.

In particular, because $a(0) = a(1)$ for $a \in LU(n)$, we see that
\begin{align*}
R_a^*\, {\rm tr}\, H(1) &= {\rm tr}\, R_a^*H(1) = {\rm tr} [a(0)^{-1}H(1)a(1)] = {\rm tr}\, H(1),
\end{align*}
proving ${\rm Tr}\, H(1)$ is $LU(n)$-invariant.  Moreover it is clear that ${\rm tr}\, H(1)$ is a horizontal form, because $\Omega$ and $\omega(X)$ are horizontal.  Therefore ${\rm Tr}\, H(1) = \tilde{\pi}^*\beta$ for some $\beta \in \Lambda^*_{S^1}(LM)$ and we define $\beta$ to be ${\rm BCh}$.  It follows from the first equation in (\ref{BismutsEqns}) that ${\rm BCh}$ is closed with respect to the differential $d + i_X$, so ${\rm BCh}$ defines a cohomology class in $\bar{h}^*_{S^1}(LM)$.

In this way Bismut modifies the construction of the $S^1$-equivariant Chern character to produce an $LU(n)$-invariant differential form, salvaging the techniques of $S^1$-equivariant Chern-Weil theory that fail because the $S^1$ and $LU(n)$ actions on $LFrE$ do not commute.  On the other hand, Bismut's construction does not define an element of the Cartan model and so it does not define a class in $H_{S^1}^*(LM)$.  For if we introduce $u$ by the integral equation
\begin{align*}
\tilde{H}(t) &= Id + \int_0^t \tilde{H}(v)(\Omega_{\gamma(v)} - u \omega(X) )dv,
\end{align*}
whose solution $\tilde{H}(t)$ is a power series of even-degree $S^1$-equivariant differential forms valued in $L\mathfrak{u}(n)$,  we find that ${\rm tr}\, \tilde{H}(1)$ defines an equivariantly closed differential form on $LFrE$ that is not basic because

\begin{proposition}
${\rm tr}\, \tilde{H}(1)$ is not $LU(n)$-invariant.
\end{proposition}

\begin{proof}
Let $A(t) = \Omega_{\gamma(t)} - u \omega(X)$.  Then $\tilde{H}(t)$ can be written
\begin{align*}
\tilde{H}(t) &= Id + \sum_{m \geq 1} \int_0^t \int_0^{t_m} \cdots \int_0^{t_2} A(t_m) \cdots A(t_1) dt_1 \ldots dt_m.
\end{align*}
Therefore the degree 2 component of ${\rm tr}\, \tilde{H}(1)$ is given by
\begin{align*}
{\rm tr}\, \tilde{H}(1)_{[2]} &= \int_0^1 {\rm tr} (\Omega_{\gamma(t)} - u \omega(X))dt = {\rm Tr}(\tilde{\Omega} - u\tilde{\omega}(X) ).
\end{align*}
Equation (\ref{LUnTransgression}) then shows that ${\rm tr}\, \tilde{H}(1)_{[2]}$ is not $LU(n)$-invariant.
\end{proof}

\begin{remark}
In \cite{GJP}, another version of the Bismut-Chern character is constructed by the integral equation
\begin{align*}
H(t) &= Id + \int_0^t H(v)(-u^{-1}\Omega_{\gamma(v)}  + \omega(X) ) dv,
\end{align*}
producing a class in $h_{S^1}^*(LM)$.  This version of the Bismut-Chern character is the image of ${\rm BCh}$ under $q_0^{-1} : \bar{h}_{S^1}^{ev}(LM) \to h_{S^1}^0(LM)$.  In \cite{TWZ} this Bismut-Chern character is used to define a refinement of differential K-theory.
\end{remark}

\section{$S^1$-equivariant first Chern class}\label{EquFCC}

Equation (\ref{LUnTransgression}) suggests ${\rm Tr}(\widetilde{\Omega} - u \tilde{\omega}(X))$ is invariant under those elements $a \in LU(n)$ satisfying $W(\det a) = 0$.  These elements make up $L^0U(n)$, the connected component of $LU(n)$ containing the identity, implying ${\rm Tr}(\widetilde{\Omega} - u \tilde{\omega}(X))$ descends to $LM$ whenever $LFrE$ admits a reduction of its structure group to this subgroup.  This section explores this special class of vector bundles and defines an $S^1$-equivariant first Chern class on these bundles with $S^1$-equivariant Chern-Weil techniques.

Section \ref{StructureGroupReduction} proves a criterion that determines when $LFrE$ admits such a reduction of its structure group and offers some examples.  Section \ref{ReducedCircleAction} then defines an $S^1$-action on the reduced bundle $L^0FrE$ so that the embedding $j : L^0FrE \to LFrE$ is $S^1$-equivariant.  Finally, Section \ref{EquFCCDefn} shows that $j^*{\rm Tr}(\widetilde{\Omega} - u\tilde{\omega}(X) )$ descends to an equivariant 2-form on $LM$ defining a class we call the $S^1$-equivariant first Chern class of $\mathcal{E}$, $c_1^{S^1}(\mathcal{E})$.  Section \ref{NontrivialExamples} concludes by providing examples of loop spaces that admit non-trivial $c_1^{S^1}(\mathcal{E})$.

\subsection{Reduction of structure group to $L^0U(n)$}\label{StructureGroupReduction}

Recall that for a Lie group $G$, the space $EG$ is defined up to homotopy equivalence by the requirements that $EG$ is contractible and $G$ acts freely on it.  We then define the classifying space of $G$ to be $BG = EG/G$.  See \cite[\textsection 1]{GS} for a proof that classifying spaces exist for compact $G$.  In particular, we may choose a realization of $EU(n)$, a contractible space on which $U(n)$ acts freely.

It follows that we may take $ELU(n) = LEU(n)$, since $LEU(n)$ is contractible and $LU(n)$ acts freely on it.  Then 
\begin{align*}
BLU(n) = LEU(n)/LU(n) \approx L(EU(n)/U(n)) = LBU(n).
\end{align*}
Let $L^0U(n)$ be the connected component of the identity in $LU(n)$.  $L^0U(n)$ acts freely on $LEU(n)$ as well, so we may take $BL^0U(n) = LEU(n)/L^0U(n)$.  Moreover, the inclusion $L^0U(n) \to LU(n)$ induces a map $p : BL^0U(n) \to BLU(n)$ given by $p[\gamma]_{L^0U(n)} = [\gamma]_{LU(n)}$.

Let $f : M \to BU(n)$ be a classifying map for the rank $n$ complex vector bundle $E \to M$.  Then $\tilde{f} : LM \to LBU(n) = BLU(n)$ is a classifying map for the pushdown bundle $\mathcal{E} \to LM$, where $\tilde{f}(\gamma) = f \circ \gamma$.  We wish to characterize when $LFrE$ admits a reduction of its structure group to $L^0U(n)$ and we begin with several lemmas.

\begin{lemma}\label{TorsionFree}
$H^1(LM ; \mathbb{Z})$ is torsion-free.
\end{lemma}

\begin{proof}
By the universal coefficient theorem the following sequence is exact
\begin{align*}
0 \to {\rm Ext}(H_0(LM), \mathbb{Z}) \to H^1(LM ; \mathbb{Z}) \to {\rm Hom}(H_1(LM) ; \mathbb{Z}) \to 0.
\end{align*}
${\rm Ext}(H_0(LM), \mathbb{Z}) = 0$, since $H_0(LM) \cong \oplus_\alpha \mathbb{Z}$.  Moreover, ${\rm Hom}(H_1(LM); \mathbb{Z})$ is torsion-free, implying $H^1(LM ; \mathbb{Z})$ is torsion-free as well.
\end{proof}

\begin{lemma}\label{Fiber}
For any $x \in BLU(n)$, $p^{-1}(x)$ is diffeomorphic to $\mathbb{Z}$.  Moreover, $p : BL^0U(n) \to BLU(n)$ is a covering space in the sense of \cite{Hatcher}.
\end{lemma}

\begin{proof}
We first prove that $p^{-1}(x)$ is diffeomorphic to $\mathbb{Z}$.  Let $a \in LU(n)$ such that $[a]$ generates $\pi_1 U(n)$.  Then any $b \in LU(n)$ can be written $b = \bar{b} a^k$, for some $\bar{b} \in L^0U(n)$ and $k \in \mathbb{Z}$.  It follows that if $[\gamma_1]_{LU(n)} = [\gamma_2]_{LU(n)}$, then $[\gamma_1]_{L^0U(n)} = [\gamma_2]_{L^0U(n)} \cdot a^k$, proving that $p^{-1}(x) \cong \mathbb{Z}$ for any $x \in BLU(n)$.

To prove that $p : BL^0U(n) \to BLU(n)$ is a covering space, we first observe that $EU(n) \to BU(n)$ admits the structure of a principal $U(n)$-bundle \cite[\textsection 1.1]{GS}.  It follows that $ELU(n) \to BLU(n)$ admits the structure of a principal $LU(n)$-bundle \cite[\textsection 4.3]{Stacey}.  

We note that \cite[Theorem 4.6]{Stacey} does not, strictly speaking, apply in this case, since Stacey's results are stated for a finite dimensional principal $G$-bundle that is `looped' to yield a principal $LG$-bundle.  However, $EU(n)$ and $BU(n)$ admit realizations as infinite dimensional Hilbert manifolds \cite[\textsection 1.1]{GS}, and his arguments remain valid in this setting with only minor adjustments.

Let $V \subset BLU(n)$ be a neighborhood over which $ELU(n)$ is trivial, i.e. $ELU(n)|_V \approx LU(n) \times V$.  Using the isomorphism $LU(n) \cong L^0U(n) \rtimes \mathbb{Z}$ \cite[\textsection 4.7]{PS} we have $ELU(n)|_V \approx L^0U(n) \times \mathbb{Z} \times V$.  Therefore $BL^0U(n)|_V \approx \mathbb{Z} \times V$, and $p$ is given by the projection on the second factor $\mathbb{Z} \times V \to V$.  Therefore $p : BL^0U(n) \to BLU(n)$ is a covering space.
\end{proof}

\begin{lemma}\label{FundamentalGroups}
$\pi_1BLU(n) \cong \mathbb{Z}$ and $\pi_1BL^0U(n) \cong  1$.
\end{lemma}

\begin{proof}
To prove the first claim, we consider the principal $LU(n)$-bundle $ELU(n) \to BLU(n)$.  We may apply the long exact sequence of homotopy groups to this fibration and we have
\begin{align*}
\ldots \to \pi_kELU(n) \to \pi_kBLU(n) \to \pi_{k - 1}LU(n) \to \pi_{k - 1}ELU(n) \to \ldots
\end{align*}
Since $ELU(n)$ is contractible, this sequence implies $\pi_k BLU(n) \cong \pi_{k - 1}LU(n)$.  In particular,
\begin{align*}
\pi_1 BLU(n) \cong \pi_0 LU(n) = \pi_0 [\Omega U(n) \times U(n)] = \pi_1U(n) \times \pi_0U(n)  \cong \mathbb{Z}.
\end{align*}
If we again apply the long exact sequence of homotopy groups to the principal $L^0U(n)$-bundle $EL^0U(n) \to BL^0U(n)$ we similarly see $\pi_k BL^0U(n) \cong \pi_{k - 1}L^0U(n)$.  Therefore $\pi_1BL^0U(n)  \cong \pi_0L^0U(n) \cong \{ 1 \}$.
\end{proof}

\begin{lemma}\label{InducedMaps}
Let $\tilde{f} : LM \to BLU(n)$.  Then $\tilde{f}_* : \pi_1 LM \to \pi_1 BLU(n)$ is trivial if and only if $\tilde{f}_* : H_1(LM) \to H_1(BLU(n))$ is the zero map.
\end{lemma}

\begin{proof}

Let $\tilde{f}_*^{\pi_1}$ be the induced map on fundamental groups and let $\tilde{f}_*^{H_1}$ be the induced map on homology groups.  Also let $h_1 : \pi_1 LM \to H_1(LM)$ and $h_2 : \pi_1 BLU(n) \to H_1( BLU(n))$ be the homomorphisms obtained by regarding loops as singular 1-cycles \cite[\textsection 2.A]{Hatcher}.  We have the following diagram
\begin{align*}
\begin{CD}
\pi_1 LM@> \tilde{f}_*^{\pi_1} >> \pi_1 BLU(n)\\
@VV h_1 V @VV h_2 V\\
H_1(LM) @> \tilde{f}_*^{H_1} >> H_1(BLU(n))
\end{CD}
\end{align*}
Moreover, this diagram commutes because $h_2 \circ \tilde{f}_*^{\pi_1} [\gamma] = \tilde{f}_*^{H_1} \circ h_1 [\gamma] = [\tilde{f} \circ \gamma] \in H_1(BLU(n))$.  Note that $h_1$ and $h_2$ are surjective, $\ker h_1$ is the commutator subgroup of $\pi_1 LM$, and $\ker h_2$ is the commutator subgroup of $\pi_1 BLU(n)$.  It follows that $h_2$ is an isomorphism, since $\pi_1BLU(n)$ is abelian.

Suppose $\tilde{f}_*^{\pi_1} = 0$.  Then
\begin{align*}
0 &= h_2 \circ \tilde{f}_*^{\pi_1} = \tilde{f}_*^{H_1} \circ h_1.
\end{align*}
Because $h_1$ is surjective, this implies $\tilde{f}_*^{H_1} = 0$ as well.  On the other hand, suppose $\tilde{f}_*^{H_1} = 0$.  Then
\begin{align*}
0 &= \tilde{f}_*^{H_1} \circ h_1 = h_2 \circ \tilde{f}_*^{\pi_1}.
\end{align*}
Because $h_2$ is an isomorphism, it must be that $\tilde{f}_*^{\pi_1} = 0$.  We have shown $\tilde{f}_*^{\pi_1} = 0$ if and only if $\tilde{f}_*^{H_1} = 0$, proving our claim.
\end{proof}

The \emph{transgression map on cohomology} $\tau^* : H^k(M; \mathbb{Z}) \to H^{k - 1}(LM ; \mathbb{Z})$ is defined as the composition
\begin{align*}
H^k(M ; \mathbb{Z}) \overset{ev^*}{\to} H^k(LM \times S^1 ; \mathbb{Z}) \to H^{k - 1}(LM ; \mathbb{Z}),
\end{align*}
where the second arrow is given by the slant product with the generator of $H_1(S^1)$ \cite[\textsection 2]{KK}.  The transgression map is also defined for cohomology with $\mathbb{C}$-coefficients, where the second arrow is given by integration over $S^1$-fibers.  The transgression map plays an important role in the cohomology of loop spaces.  Recall that $H^*(BU(n); \mathbb{Z}) \cong \mathbb{Z}[c_1, \ldots, c_n]$, where $f^*c_j = c_j(E)$ whenever $f : M \to BU(n)$ is a classifying map for $E \to M$.  Proposition 3 of \cite{KK} states that $H^*(LBU(n); \mathbb{Z}) \cong \mathbb{Z}[c_1, \ldots, c_n] \otimes \Lambda (\tau^* c_1, \ldots, \tau^* c_n)$, where $\Lambda( \tau^* c_1, \ldots \tau^* c_n)$ is the exterior algebra generated by $\tau^* c_1, \ldots, \tau^* c_n$.  It follows that $H^1(BLU(n); \mathbb{Z})$ is generated by $\tau^* c_1$.

\begin{proposition}\label{ReductionCriterion}
$\mathcal{E}$ admits a reduction of its structure group to $L^0U(n)$ if and only if $\tau^* c_1(E) = 0$.
\end{proposition}

\begin{proof}

$\mathcal{E}$ admits a reduction of its structure group if and only if $\tilde{f} : LM \to BLU(n)$ admits a lift $\hat{f} : LM \to BL^0U(n)$.  
\begin{center}
\begin{tikzcd}
&BL^0U(n) \arrow{d}{p} \\
LM \arrow{r}[swap]{\tilde{f}}\arrow{ru}{\hat{f}}
&BLU(n)
\end{tikzcd}
\end{center}
Since $p : BL^0U(n) \to BLU(n)$ is a covering space in the sense of \cite{Hatcher}, this lift exists if and only if $\tilde{f}_*(\pi_1 LM) \subset p_* (\pi_1 BL^0U(n))$ \cite[Proposition 1.33]{Hatcher}.  By Lemma \ref{FundamentalGroups}, $\pi_1 BL^0U(n) = \{ 1 \}$.  Therefore $\tilde{f}$ admits a lift if and only if $\tilde{f}_* : \pi_1 LM  \to \pi_1 BLU(n)$ is trivial.  By Lemma \ref{InducedMaps}, this is true if and only if $\tilde{f}_* : H_1(LM) \to H_1(BL^0U(n))$ is the zero map.

We will show that $\tilde{f}_* = 0$ if and only if $\tau^* c_1(E) = 0$.  We first remark that $\tau^* c_1(E) = \tau^* f^* c_1 = \tilde{f}^* \tau^* c_1$ \cite[Prop. 2]{KK}.  Suppose $\tilde{f}_* = 0$ and let $\sigma \in H_1(LM)$.  Then
\begin{align*}
\langle \tilde{f}^* \tau^* c_1, \sigma \rangle = \langle \tau^* c_1, \tilde{f}_* \sigma \rangle = 0,
\end{align*}
since $\tilde{f}_* \sigma = 0$.  Because $H^1(LM; \mathbb{Z}) \cong {\rm Hom}_{\mathbb{Z}}(H_1(LM), \mathbb{Z})$, we have proven $\tilde{f}^*\tau^* c_1 = \tau^* c_1(E) = 0$.

On the other hand, suppose $\tilde{f}^* \tau^* c_1 = 0$ and let $\sigma \in H_1(LM)$.  Then
\begin{align}\label{PairingVanishes}
0 &= \langle \tilde{f}^* \tau^* c_1, \sigma \rangle = \langle \tau^* c_1, \tilde{f}_* \sigma \rangle.
\end{align}
Since $H_1(BLU(n)) \cong \mathbb{Z}$, it is generated by some element $\alpha$ and we may write $\tilde{f}_* \sigma = k \alpha$.  Since $\tau^* c_1$ generates $H^1(BLU(n); \mathbb{Z})$, we may pick $\alpha$ so that $\langle \tau^* c_1, \alpha \rangle = 1$.  Therefore
\begin{align*}
\langle \tau^* c_1, \tilde{f}_* \sigma \rangle =  \langle \tau^* c_1, k \alpha \rangle = k.
\end{align*}
Combining this with equation (\ref{PairingVanishes}), we see $k = 0$ and consequently $\tilde{f}_* = 0$.  Therefore $LFrE$ admits a reduction of its structure group to $L^0U(n)$ if and only if $\tau^* c_1(E) = 0$.
\end{proof}

Suppose $LFrE$ admits this reduction of structure group and suppose $\hat{\pi} : P \to LM$ is our reduced bundle.  Let $j : P \to LFrE$ be the map including $P$ as a sub-bundle of $LFrE$.

\begin{center}
\begin{tikzcd}
P \arrow[hook]{r}{j}\arrow{rd}[swap]{\hat{\pi}}
&LFrE \arrow{d}{\tilde{\pi}}\\
&LM
\end{tikzcd}
\end{center}

\begin{lemma}
Let $b \in L^0U(n)$ and $a \in LU(n)$ such that $[a]$ generates $\pi_1 U(n)$.  Then there is some $\tilde{b} \in L^0U(n)$ such that $ab = \tilde{b}a$.
\end{lemma}

\begin{proof}
This claim is equivalent to the statement that $aba^{-1} \in L^0U(n)$.  To prove this claim, it suffices to show that $[aba^{-1}]$ is the identity element in $\pi_1 U(n)$.  Since $\det : U(n) \to S^1$ induces an isomorphism on fundamental groups, we must show that $[\det (aba^{-1})]$ is trivial in $\pi_1 S^1$.  Its winding number is
\begin{align*}
W( \det(aba^{-1}) ) = W (\det b) = 0.
\end{align*}
The second equality holds because $b \in L^0U(n)$.  Therefore $aba^{-1} \in L^0U(n)$.
\end{proof}

\begin{proposition}
$R_a(j(P))$ admits the structure of a principal $L^0U(n)$-bundle over $LM$ that is also a sub-bundle of $LFrE$.
\end{proposition}

\begin{proof}
We define the projection $\pi_a : R_a(j(P)) \to LM$ by $\pi_a \overset{\rm def}{=} \tilde{\pi}|_{R_a(j(P))}$.  We must verify  that 1.) $L^0U(n)$ acts freely on $R_a(j(P))$, 2.) $R_a(j(P))/L^0U(n)$ is diffeomorphic to $LM$ and $\pi_a$ is smooth, and 3.) $R_a(j(P))$ is locally trivial.

We first show that $L^0U(n)$ acts on $R_a(j(P))$.  Let $b \in L^0U(n)$ and $x \in R_a(i(P))$.  Then $x = \tilde{x}\cdot a^{-1}$, for some $\tilde{x} \in j(P)$, and
\begin{align*}
R_b(x) &= \tilde{x}\cdot a^{-1}b^{-1} = \tilde{x} \cdot \tilde{b}^{-1}a^{-1} = R_a( R_{\tilde{b}}(\tilde{x})).
\end{align*}
Since $R_{\tilde{b}}(\tilde{x}) \in j(P)$, it follows that $R_b(x) \in R_a(j(P))$.  Therefore $L^0U(n)$ acts on $R_a(j(P))$.  Moreover, because $LU(n)$ acts freely on $LFrE$, the $L^0U(n)$-action on $R_a(j(P))$ is free as well, verifying 1.).

To verify 2.), let $\gamma_0 \in LM$ and suppose $x, y \in \pi_a^{-1}(\gamma_0)$.  Then $x = \tilde{x} \cdot a^{-1}$ and $y  = \tilde{y} \cdot a^{-1}$ for some $\tilde{x}, \tilde{y} \in j(P)$.  Moreover,
\begin{align*}
\hat{\pi}(\tilde{x}) = \tilde{\pi}(\tilde{x}) = \tilde{\pi}( \tilde{x}\cdot a^{-1}) = \tilde{\pi}(x) = \gamma_0,
\end{align*}
and similarly $\hat{\pi}(\tilde{y}) = \gamma_0$.  Therefore $\tilde{x} = \tilde{y} b$ for some $b \in L^0U(n)$, as $P$ is a principal $L^0U(n)$-bundle.  Therefore
\begin{align*}
x = \tilde{x}\cdot a^{-1} = \tilde{y} \cdot b a^{-1} = \tilde{y} \cdot a^{-1} \tilde{b} = y \cdot \tilde{b},
\end{align*}
since $ba^{-1} = a^{-1} \tilde{b}$ for some $\tilde{b} \in L^0U(n)$.  In particular, we have shown $x = y \cdot \tilde{b}$, proving that any two elements in $\pi_a^{-1}(\gamma_0)$ differ by an element of $L^0U(n)$.  Therefore $R_a(j(P))/L^0U(n) \cong LM$.  Moreover, $\pi_a$ is smooth because it is the restriction of the smooth map $\tilde{\pi}$.

To verify 3.) let $\gamma_0 \in LM$.  There exists some neighborhood $U \subset LM$ such that $LFrE|_U \approx U \times LU(n)$.  Using the isomorphism $LU(n) \cong L^0U(n) \rtimes \mathbb{Z}$ \cite[\textsection 4.7]{PS}, we may take this trivialization $LFrE|_U \approx U \times L^0U(n) \times \mathbb{Z}$.  Similarly, $P|_U \approx U \times L^0FrE$ and we may arrange these trivializations so that the inclusion $j : P|_U \to LFrE|_U$ is the map $U \times L^0U(n) \to U \times L^0U(n) \times \mathbb{Z}$ is given by $(\gamma, b) \mapsto (\gamma, b, 0)$. Thus $j(P)|_U \approx U \times L^0U(n) \times \{0 \}$, and $R_a(j(P))|_U \approx U \times \mathbb{Z} \times \{1\}$, proving that $R_a(j(P))$ is locally trivial.
\end{proof}

\begin{corollary}
When $LFrE$ admits a reduction of its structure group to $L^0U(n)$, $LFrE$ is the disjoint union of a countable collection of principal $L^0U(n)$-bundles over $LM$.
\end{corollary}

It follows that we have countably many $L^0U(n)$-bundles to choose for our reduced bundle, indexed by $\pi_1 U(n) \cong \mathbb{Z}$.  However, under the inclusion $FrE \hookrightarrow LFrE$ taking a point to the constant loop based there, only one such $L^0U(n)$-bundle contains $FrE$.  Hence only one $L^0U(n)$-bundle contains the fixed point set of the $S^1$-action on $LFrE$.  This distinguishes a canonical choice of reduced bundle, which we denote $L^0FrE$.

In the next sections we study $S^1$-equivariant Chern-Weil techniques on pushdown bundles admitting this structure group reduction.  We end this section with two examples of such pushdown bundles.

\begin{example}\label{SUnBundles}
Suppose $E \to M$ is a complex bundle with $c_1(E) = 0$.  Then $\tau^* c_1(E) = 0$ and $LFrE$ admits a reduction of its structure group to $L^0U(n)$.

In this case, one may prove directly that $LFrE$ admits this reduction of its structure group.  For if $c_1(E) = 0$, $FrE$ admits a reduction of its structure group to $SU(n)$.  Call the reduced bundle $SFrE \to M$, a principal $SU(n)$-bundle over $M$.  Then $LSFrE \to LM$ is a sub-bundle of $LFrE$ whose structure group is $LSU(n)$.  Because $LSU(n) \subset L^0U(n)$, $LSFrE$ is our desired reduced bundle.
\end{example}

\begin{example}\label{ProjectiveSpaceBundles}
We present a sequence of a non-trivial bundles $E_k \to \mathbb{RP}^k$ such that $LFrE_k \to L\mathbb{RP}^k$ admits a reduction of its structure group to $L^0U(n)$.  We first note that $H^2( \mathbb{R P}^k ; \mathbb{Z} ) = \mathbb{Z}/2 \mathbb{Z}$ for $k \geq 2$.  Since the second cohomology group with $\mathbb{Z}$ coefficients parameterizes complex line bundles over a manifold, there exists a complex line bundle $E_k \to \mathbb{RP}^k$ such that $c_1(E_k) \neq 0$ but $2c_1(E_k) = 0$.  

It follows that $2 \tau^* c_1(E_k) = \tau^*(2c_1(E_k)) = 0$ in $H^1(L\mathbb{RP}^k ; \mathbb{Z})$.  By Lemma \ref{TorsionFree}, $H^1(L\mathbb{RP}^k; \mathbb{Z})$ is torsion-free, implying $\tau^* c_1(E_k) = 0$.  Therefore by Proposition \ref{ReductionCriterion}, $LFrE_k \to L\mathbb{RP}^k$ admits a reduction of its structure group to $L^0U(n)$.

We remark that when $k$ is odd, $\mathbb{RP}^k$ is closed and orientable, so we have a non-trivial class of bundles that admit this reduction of structure group even if we only consider closed and orientable manifolds.
\end{example}

\subsection{$S^1$-action on $L^0FrE$}\label{ReducedCircleAction}

Suppose $LFrE$ admits a reduction of its structure group to $L^0U(n)$ and let $L^0FrE$ be the reduced bundle.  That is, $\hat{\pi} : L^0FrE \to LM$ is a principal $L^0U(n)$-bundle over $LM$ such that there is an embedding $j : L^0FrE \to LFrE$ which satisfies $\hat{\pi} = \tilde{\pi} \circ j$ and $j \circ R_a = R_a \circ j$ for $a \in L^0U(n)$.

\begin{center}
\begin{tikzcd}
L^0FrE \arrow[hook]{r}{j}\arrow{rd}[swap]{\hat{\pi}}
&LFrE \arrow{d}{\tilde{\pi}}\\
&LM
\end{tikzcd}
\end{center}

\begin{lemma}\label{InvariantImage}
For all $\theta \in S^1$, $k_\theta( j(L^0FrE)) = j(L^0FrE)$.
\end{lemma}

\begin{proof}
Suppose $U \subset LFrE$ is a connected component and let $\theta \in S^1$.  We claim $k_\theta(U) = U$.  Let $x \in U$.  Then the path $t \mapsto k_{t \theta}(x)$, $0 \leq t \leq 1$, joins $x$ to $k_\theta(x)$.  Because $k_\theta$ is a diffeomorphism $k_\theta(U)$ is a connected component, implying $x \in k_\theta(U)$.  Therefore $k_\theta(U) = U$.

Because $j : L^0FrE \to LFrE$ is a local diffeomorphism, $j(L^0FrE)$ is an open submanifold and we can write $j(L^0FrE) = \amalg U_\alpha$ for connected components $U_\alpha \subset LFrE$.  Since $k_\theta(U_\alpha) = U_\alpha$ for all $\alpha$, it follows that $k_\theta(j(L^0FrE)) = j(L^0FrE)$.
\end{proof}

In light of Lemma \ref{InvariantImage}, we define the following $S^1$-action on $L^0FrE$.  Given $\theta \in S^1$ and $x \in L^0FrE$, 
\begin{align*}
\hat{k}_\theta(x) \overset{\rm def}{=} j^{-1} k_\theta j(x).
\end{align*}
$\hat{k}$ is indeed a group action, for if $\theta_1, \theta_2 \in S^1$,
\begin{align*}
\hat{k}_{\theta_1} \circ \hat{k}_{\theta_2}(x) = j^{-1} k_{\theta_1} j (j^{-1} k_{\theta_2} j(x) ) = j^{-1} k_{\theta_1 + \theta_2} j(x) = \hat{k}_{\theta_1 + \theta_2}(x).
\end{align*}

\begin{proposition}
$L^0FrE$ admits an $S^1$-action such that $j : L^0FrE \to LFrE$ is an $S^1$-equivariant map.
\end{proposition}

\begin{proof}
To prove this proposition, we need only show that $i$ is an $S^1$-equivariant map with respect to these $S^1$-actions.  A straightforward calculation shows
\begin{align*}
j (\hat{k}_\theta(x)) &= j( j^{-1} k_\theta j(x)) = k_\theta j(x),
\end{align*}
proving $j$ is $S^1$-equivariant.
\end{proof}

\begin{corollary}
$j : L^0FrE \to LFrE$ induces a map $j^* : H_{S^1}^*(LFrE) \to H_{S^1}^*(L^0FrE)$.
\end{corollary}

\subsection{$S^1$-equivariant first Chern class}\label{EquFCCDefn}

Consider $j^* {\rm Tr}(\widetilde{\Omega} - u \tilde{\omega}(X))$, an equivariant 2-form on $L^0FrE$.

\begin{proposition}
$i^*{\rm Tr}(\widetilde{\Omega} - u \tilde{\omega}(X))$ is basic.
\end{proposition}

\begin{proof}
We must show that $j^* {\rm Tr}(\widetilde{\Omega} - u \tilde{\omega}(X))$ is horizontal and $L^0U(n)$-invariant.  We first show it is horizontal.  Let $v \in T_xL^0FrE$ be a vertical vector.  Then $j_*v$ is a vertical vector as well, since $L^0FrE$ is a sub-bundle, and we have
\begin{align*}
\iota_v j^*{\rm Tr}\, \widetilde{\Omega}(w) = j^*{\rm Tr}\, \widetilde{\Omega}(v, w) = {\rm Tr}\, \widetilde{\Omega}(j_*v, j_*w) = 0,
\end{align*}
since ${\rm Tr}\, \widetilde{\Omega}$ is horizontal.  Moreover, $\iota_v j^*{\rm Tr}\, \tilde{\omega}(X) = 0$ by definition.  Therefore $\iota_v j^*{\rm Tr}(\widetilde{\Omega} - u \tilde{\omega}(X)) = 0$, proving $j^*{\rm Tr}(\widetilde{\Omega} -u \tilde{\omega}(X))$ is horizontal.

Next we show it is $L^0U(n)$-invariant.  Let $a \in L^0U(n)$.  Then $W(\det a) = 0$ because the winding number of a loop is homotopy-invariant.  Since $j \circ R_a = R_a \circ j$ we have
\begin{align*}
R_a^* j^* {\rm Tr}(\widetilde{\Omega} - u \tilde{\omega}(X)) = j^* R_a^* {\rm Tr}(\widetilde{\Omega} - u \tilde{\omega}(X)) = j^* {\rm Tr}(\widetilde{\Omega} - u \tilde{\omega}(X)),
\end{align*}
by equation (\ref{LUnTransgression}).  Therefore $j^*{\rm Tr}(\widetilde{\Omega} - u \tilde{\omega}(X))$ is $L^0U(n)$-invariant.
\end{proof}

It follows that $j^*{\rm Tr}(\widetilde{\Omega} - u \tilde{\omega}(X)) = \hat{\pi}^* \beta$, for some equivariant 2-form $\beta$ on $LM$.

\begin{mydef}
The $S^1$-equivariant first Chern class of $\mathcal{E}$ is $c_1^{S^1}(\mathcal{E}) \overset{\rm def}{=} [\beta] \in H_{S^1}^2(LM)$, and $\beta$ is the $S^1$-equivariant first Chern form.
\end{mydef}
This $S^1$-equivariant characteristic class is notable because it admits an $S^1$-equivariant Chern-Weil construction.

Let $i : M \to LM$ be the embedding sending $x$ to the constant loop based at $x$.  Standard Chern-Weil constructions show that ${\rm tr}\, \Omega$ is a basic form on $FrE$, implying ${\rm tr}\, \Omega = \hat{\beta}$ for some closed 2-form $\hat{\beta}$ on $M$ and $[\hat{\beta}] = c_1(E)$.

\begin{theorem}\label{FCCExtension}
$c_1^{S^1}(\mathcal{E})$ extends $c_1(E)$ under the embedding $M \hookrightarrow LM$.  Moreover, the $S^1$-equivariant first Chern form of $\mathcal{E}$ and $\tilde{\omega}$ extends the ordinary first Chern form of $E$ and $\omega$.
\end{theorem}
 
\begin{proof}
To prove this theorem, it suffices to show that $i^*\beta = \hat{\beta}$.  Recall that we may write $\beta = \beta_{[2]} + u \beta_{[0]}$ for differential forms $\beta_{[k]}$ of degree $k$.

Suppose $v, w \in T_xM$.  We may take lifts $\bar{v}, \bar{w} \in T_yFrE$ for some $y \in FrE$ such that $\pi(y) = x$.  The vectors $i_*v, i_*w$ are ``constant'' in the sense that they are represented by the constant loops $t \mapsto v$ and $t \mapsto w$ in $TLM = LTM$.  Similarly, we may lift $i_*v$ and $i_*w$ to the constant loops $t \mapsto \bar{v}$ and $t \mapsto \bar{w}$, which we write $\overline{i_*v}$ and $\overline{i_*w}$.  Then
\begin{align*}
i^* \beta_{[2]}(v, w) &= \beta_{[2]}(i_*v, i_*w) = j^*{\rm Tr}\, \widetilde{\Omega}( \overline{i_*v}, \overline{i_*w} )
\end{align*}
The vectors $j_* \overline{i_*v}$ and $j_* \overline{i_*w}$ are also ``constant'' in the sense that they are represented by the same constant loops $t \mapsto \bar{v}$ and $t \mapsto \bar{w}$, following from the fact that $j_* : T_xL^0FrE \to T_{i(x)}LFrE$ is an isomorphism.  Moreover,
\begin{align*}
j^*{\rm Tr}\, \widetilde{\Omega}( \overline{i_*v}, \overline{i_*w} ) &= {\rm Tr}\, \widetilde{\Omega}(j_* \overline{i_*v}, j_* \overline{i^* w}) = \frac{1}{2\pi} \int_{S^1} {\rm tr}\, \Omega(\bar{v}, \bar{w}) dt\\
&= {\rm tr}\, \Omega(\bar{v}, \bar{w}) = \hat{\beta}(v, w).
\end{align*}
Therefore $i^* \beta_{[2]} = \hat{\beta}$.

Next, let $x \in M$.  Then $i(x)$ is a constant loop, which for notational ease we call $x$.  We may lift $x$ to $y \in L^0FrE$.  Thus $j(y)$ is a constant loop in the fiber over the constant loop $x$, so that its velocity vector field $\frac{d}{dt} j(y)$ vanishes, and
\begin{align*}
i^* \beta_{[0]}(x) &= -j^*{\rm Tr}\, \tilde{\omega}\left(\frac{d}{dt} y \right) = -{\rm Tr}\, \tilde{\omega}\left( \frac{d}{dt} j(y) \right) = 0.
\end{align*}
Therefore
\begin{align*}
i^*\beta = i^*\beta_{[2]} + u i^* \beta_{[0]} = \hat{\beta}.
\end{align*}
\end{proof}

In summary, when $\tau^* c_1(E) = 0$ in $H^1(LM; \mathbb{Z})$, $S^1$-equivariant Chern-Weil techniques define $c_1^{S^1}(\mathcal{E})$, an $S^1$-equivariant first Chern class extending $c_1(E) \in H^2(M)$ to $H_{S^1}^2(LM)$.  The following criterion may detect bundles for which this class does not vanish.

\begin{proposition}\label{NontrivialFCCCriterion}
Suppose $E \to M$ is a complex bundle such that $\tau^* c_1(E) = 0$ and $c_1(E) \neq 0$ in $H^2(M ; \mathbb{C})$.  Then $c_1^{S^1}(\mathcal{E}) \neq 0$.
\end{proposition}

\begin{proof}
By Theorem \ref{FCCExtension}, $i^*c_1^{S^1}(\mathcal{E}) = c_1(E) \neq 0$, implying $c_1^{S^1}(\mathcal{E}) \neq 0$.
\end{proof}

By this proposition, to find a non-trivial instance of $c_1^{S^1}(\mathcal{E})$ we need only find a bundle $E \to M$ with non-torsion $c_1(E) \in H^2(M ; \mathbb{Z})$ belonging to the kernel of $\tau^* : H^2(M ; \mathbb{Z}) \to H^1(LM ; \mathbb{Z})$.  In the next section we point out a collection of such bundles.

\subsection{$c_1^{S^1}(\mathcal{E})$ on the loop space of a Riemann surface}\label{NontrivialExamples}

In this section we study the transgression map on compact Riemann surfaces and we show that, when $g \geq 2$, the loop space of $\Sigma_g$, the compact Riemann surface of genus $g$, admits pushdown bundles with non-trivial $c_1^{S^1}(\mathcal{E})$.  To prove this result we consider the \emph{transgression map on homology}, $\tau_* : H_1(LM) \to H_2(M)$, defined as the composition
\begin{align*}
H_1(LM) \to H_2(LM \times S^1) \overset{ev_*}{\to} H_2(M),
\end{align*}
where the first arrow is given by the homology cross product with the generator of $H_1(S^1)$.  Our two transgression maps enjoy the adjoint property
\begin{align*}
\langle \tau^* \omega, \beta \rangle = \langle \omega, \tau_* \beta \rangle.
\end{align*}
We begin with two lemmas.

\begin{lemma}\label{TorusCycle}
For any manifold $M$, if $\beta$ is a loop on $LM$, then $\tau_* \beta$ is a 2-cycle on $M$ given by a map $T^2 \to M$.
\end{lemma}

\begin{proof}
Since $\beta$ is a loop, $\beta : [0, 1] \to LM$ such that $\beta(0) = \beta(1)$.  Thus, for each $t \in [0, 1]$, $\beta(t) \in LM$ and we let $s \in [0, 1]$ denote the ``loop parameter.''  We define $\hat{\beta} : [0,1 ] \times [0,1] \to M$ by $\hat{\beta}(t, s) \overset{\rm def}{=} \beta(t)(s)$.  We claim that $\hat{\beta}(t, s)$ descends to a map $T^2 \to M$.  Notice that 
\begin{align*}
\hat{\beta}(0, s) = \hat{\beta}(1, s)
\end{align*}
because $\beta$ is a loop on $LM$.  Moreover,
\begin{align*}
\hat{\beta}(t, 0) = \hat{\beta}(t, 1)
\end{align*}
since $\beta(t)$ is a loop in $M$ for each $t \in [0, 1]$.  Therefore $\hat{\beta}(t, s)$ descends to a map $T^2 \to M$.

Recall that $\tau_* \beta = ev \circ (\beta \times Id_{S^1})$.  That is, if $s \in S^1$ and $t \in [0, 1]$, 
\begin{align*}
\tau_* \beta (t, s) &= ev( \beta(t), s) = \beta(t)(s) = \hat{\beta}(t, s).
\end{align*}
Therefore $\tau_* \beta = \hat{\beta}$, and so $\tau_* \beta$ is a 2-cycle given by a map $T^2 \to M$.
\end{proof}

\begin{lemma}\label{DegreeTorusMap}
Suppose $f : T^2 \to \Sigma_g$, with $g \geq 2$.  Then $\deg f = 0$.
\end{lemma}

\begin{proof}
Let $f : T^2 \to \Sigma_g$ and suppose $f^* : H^2(\Sigma_g ; \mathbb{Z}) \to H^2(T^2 ; \mathbb{Z})$ is nonzero.  We claim that $f^* : H^1(\Sigma_g ; \mathbb{Z}) \to H^1(T^2 ; \mathbb{Z})$ has rank 2.  For we may take $\omega_1, \omega_2 \in H^1(\Sigma_g ; \mathbb{Z})$ such that $\omega_1 \smile \omega_2 \neq 0$, and so
\begin{align*}
f^*\omega_1 \smile f^* \omega_2 = f^*( \omega_1 \smile  \omega_2) \neq 0,
\end{align*}
proving $f^*\omega_1$ and $f^* \omega_2$ are linearly independent.  Therefore $f^* : H^1(\Sigma_g ; \mathbb{Z}) \to H^1(T^2 ; \mathbb{Z})$ has rank 2.

Moreover, we claim that $f_* : \pi_1 T^2 \to \pi_1 \Sigma_g$ is injective.  For let $\gamma \in \pi_1 T^2$ be nonzero.  We may pick $\eta \in H^1(T^2 ; \mathbb{Z})$ such that $\langle \eta, \gamma \rangle \neq 0$.  Because $f^* : H^1(\Sigma_g ; \mathbb{Z}) \to H^1(T^2 ; \mathbb{Z})$ has rank 2, it is surjective and we may pick $\omega \in H^1(\Sigma_g ; \mathbb{Z})$ such that $f^* \omega = \eta$. Therefore
\begin{align*}
0 &\neq \langle \eta, \gamma \rangle = \langle f^*\omega, \gamma \rangle = \langle \omega, f_* \gamma \rangle,
\end{align*}
proving $f_* \gamma \neq 0$.

However, $\pi_1 \Sigma_g$ does not contain a copy of $\mathbb{Z}^{\oplus 2} \cong \pi_1 T^2$, which is a contradiction.  Therefore  $f^* : H^2(\Sigma_g ; \mathbb{Z}) \to H^2(T^2 ; \mathbb{Z})$ is the zero map, and so $\deg f = 0$.
\end{proof}

\begin{proposition}
$\tau_* : H_1(L\Sigma_g) \to H_2(\Sigma_g)$ is zero.
\end{proposition}

\begin{proof}
Let $\beta$ be a 1-cycle on $L\Sigma_g$.  Then we may write $\beta = \sum_i a_i \beta_i$, where $a_i \in \mathbb{Z}$ and $\beta_i$ is a loop on $L\Sigma_g$.  Then $\tau_* \beta = \sum_i a_i \tau_* \beta_i$ is a 2-cycle on $\Sigma_g$ and each $\tau_* \beta_i$ is given by a map $f_i : T^2 \to \Sigma_g$, by Lemma \ref{TorusCycle}.  Therefore $\tau_* [\beta] = \sum_i  a_i {f_i}_* [T^2] \in H_2(\Sigma_g)$.  Let $\omega \in \Lambda^2(\Sigma_g)$.  Then
\begin{align*}
\langle  \omega, \tau_* \beta \rangle &= \sum_i a_i \langle \omega, {f_i}_* [T^2] \rangle = \sum_i a_i (\deg f_i) \langle \omega, [\Sigma_g] \rangle = 0,
\end{align*}
as Lemma \ref{DegreeTorusMap} guarantees $\deg f_i = 0$.  Since $\omega$ was arbitrary, we have shown $\tau_* [\beta] = 0$, proving $\tau_*  = 0$.
\end{proof}

\begin{corollary}\label{TransgressionVanishes}
$\tau^* : H^2(\Sigma_g ; \mathbb{Z}) \to H^1(L\Sigma_g ; \mathbb{Z})$ is zero.
\end{corollary}

\begin{proof}
Let $\beta$ be a 1-cycle on $L\Sigma_g$.  Then $\langle \tau^*\omega, \beta \rangle = \langle \omega, \tau_* \beta \rangle = 0$, proving that $\tau^*\omega = 0$.
\end{proof}

In particular, we may use this proposition to produce examples of loop spaces that admit non-trivial $c_1^{S^1}(\mathcal{E})$.  

\begin{proposition}
Let $E \to \Sigma_g$ be a non-trivial complex line bundle.  Then $c_1^{S^1}(\mathcal{E}) \neq 0$.
\end{proposition}

\begin{proof}
Corollary \ref{TransgressionVanishes} shows $\tau^* c_1(E) = 0$, while $c_1(E) \neq 0$ in $H^2(\Sigma_g ; \mathbb{Z})$ as $E$ is non-trivial.  Because $H^2(\Sigma_g ; \mathbb{Z}) \cong \mathbb{Z}$ is torsion-free, $c_1(E) \neq 0$ in $H^2(\Sigma_g ; \mathbb{C})$ as well.  Therefore, by Proposition \ref{NontrivialFCCCriterion}, $c_1^{S^1}(\mathcal{E}) \neq 0$.
\end{proof}

\end{document}

%% file: newheading.tex
\newtheorem{theorem}{Theorem}[section]
\newtheorem{corollary}[theorem]{Corollary}
\newtheorem{lemma}[theorem]{Lemma}
\newtheorem{proposition}[theorem]{Proposition}

\theoremstyle{definition}
\newtheorem{mydef}{Definition}[section]
\newtheorem{remark}{Remark}[section]

\newtheorem{example}[theorem]{Example}

\renewcommand{\to}{\longrightarrow}

\newsavebox{\savepar}

\numberwithin{equation}{section}

\newcounter{labelflag} 
\newcommand{\labelon}{\setcounter{labelflag}{1}}
\newcommand{\Label}[1]{
                       \ifnum\thelabelflag=1
                          \ifmmode
                             \makebox[0in][l]{\qquad\fbox{\rm#1}}
                          \else
                             \marginpar{\vspace{0.7\baselineskip}
                                        \hspace{-1.1\textwidth}
                                        \fbox{\rm#1}}
                          \fi
                       \fi
                       \label{#1}
                      }
\labelon